\documentclass[pdflatex,sn-mathphys-num]{sn-jnl}

\usepackage{graphicx}%
\usepackage{multirow}%
\usepackage{amsmath,amssymb,amsfonts}%
\usepackage{amsthm}%
\usepackage{mathrsfs}%
\usepackage[title]{appendix}%
\usepackage{xcolor}%
\usepackage{textcomp}%
\usepackage{manyfoot}%
\usepackage{booktabs}%
\usepackage{algorithm}%
\usepackage{algorithmicx}%
\usepackage{algpseudocode}%
\usepackage{listings}%
\usepackage{cleveref}
\usepackage{color}
\usepackage[normalem]{ulem} 

\theoremstyle{thmstyleone}%
\newtheorem{theorem}{Theorem}[section]

\newtheorem{example}{Example}[section]%

\newtheorem{lemma}{Lemma}[section]
\newtheorem{remark}{Remark}[section]

\begin{document}

\title[Article Title]{Adjoint-based A Posteriori Error Analysis for Semi-explicit Index-1 and Hessenberg Index-2 Differential-Algebraic Equations}


\author*[1]{\fnm{Jehanzeb H.} \sur{Chaudhry}}\email{jehanzeb@unm.edu}
\equalcont{These authors contributed equally to this work.}
\author[1]{\fnm{Owen L.} \sur{Lewis}}\email{owenlewis@unm.edu}
\equalcont{These authors contributed equally to this work.}
\author[1]{\fnm{Md Al Amin} \sur{Molla}}\email{alaminmolla@unm.edu}
\equalcont{These authors contributed equally to this work.}

\affil[1]{\orgdiv{Department of Mathematics \& Statistics}, \orgname{University of New Mexico}, \orgaddress{
\city{Albuquerque}, \postcode{87131}, \state{New Mexico}, \country{USA}}}




\abstract{In this work we develop  adjoint-based analyses for \textit{a posteriori} error estimation for the temporal discretization of differential-algebraic equations (DAEs) of special type: semi-explicit index-1 and Hessenberg index-2. Our technique quantifies the error in a Quantity of Interest (QoI), which is defined as a bounded linear functional of the solution of a DAE. We derive representations for errors of various types of QoIs (depending on the entire time interval, final time, algebraic variables, differential variables, etc.).  We develop two analyses: one that defines the adjoint to the DAE system, and one that first converts the DAE to an ODE system and then applies classical \textit{a posteriori} analysis techniques. 
A number of examples are presented, including nonlinear and non-autonomous DAEs, as well as spatially discretized partial differential-algebraic equations (PDAEs). Numerical results indicate a high degree of accuracy in the error estimation.}




\maketitle

\newcommand{\jhc}[1]{\textcolor{magenta} {JHC: #1}}
\newcommand{\md}[1]{\textcolor{red} {Md: #1}}


\newcommand{\QCI}[2]{\ensuremath{\mathcal{Q}_{\left[0,T\right]}(#1,#2)}}
\newcommand{\QFT}[2]{\ensuremath{\mathcal{Q}_{T}(#1,#2)}}

\newcommand{\eQCI}{\ensuremath{e^{\mathcal{Q}_{\left[0,T\right]}}}}
\newcommand{\eQFT}{\ensuremath{e^{\mathcal{Q}_{T}}}}

\section{Introduction}
        Differential-algebraic equations (DAEs) are a broad class of  mathematical models that arise in a wide range of scientific and engineering applications where constraints are intrinsic to the system under consideration. Such systems arise in  optimization, model simplification, optimal control, uncertainty quantification, and in modeling numerous physical phenomena. 
        Examples include electrical circuits, combustion, and ion flow \cite{Chua65,Lutz1985, profowensmodel,adjointsensitivitypetzoldae,nsivpdaepetzol,solvingodestiffanddae}. 
        DAEs are a generalization of ordinary differential equations (ODEs).
        Unlike ODEs, DAEs combine both differential dynamics and algebraic constraints, which introduce unique challenges for their analysis and numerical simulation.
        In this article we focus on DAEs containing only ordinary derivatives, as opposed to partial derivatives.
        However, discretization of PDEs with constrains often leads to DAEs such as those we consider in this work \cite{nsivpdaepetzol}.  
        

        Numerical techniques for solving DAEs must handle complexities  like stiffness and nonlinearity inherent to ODEs, but also  the additional difficulties introduced by the algebraic structure, index, hidden constraints and consistency of initial conditions. The index of a DAE, which represents the complexity of its constraint structure, plays an important role in determining the appropriate numerical method. 
        There are a number of techniques for the numerical solution of DAEs (see \cref{subsec:ns} and the references therein).
        
        Solving DAEs numerically invariably involves error in the computed solution. This error in the computed solution must be quantified if DAEs are to be used in a robust and reliable manner in science and engineering applications. In this article we develop techniques to accurately quantify the error in a numerically computed Quantity-of-Interest (QoI) for DAEs.
        Adjoint based \textit{a posteriori} error estimation plays a key role in accurate estimation of the error in the computed value of a QoI~\cite{ODE_Eriksson_Estep_Hansbo_Johnson_1995,Prudhomme1999,becker2001optimal,Giles2002}. Adjoint based methods and analyses are used in numerous other applications, e.g.,   sensitivity analysis \cite{adjointsensitivitypetzoldae,cheng2020surrogate,serban2005cvodes,CAO2002171}, optimization \cite{hammond2021photonic,zhang2024physics,su2023kinetics}, optimal control \cite{manzoni2021optimal,gerdts2024optimal}, large-scale DAE solver: SUNDIALS \cite{hindmarsh2005sundials,gardner2022enabling}, neural network \cite{kharazmi2021hp,haber2017stable,tahersima2019deep}, and uncertainty quantification \cite{MITRA20199}.

    Although a well-developed theoretical framework for adjoint based \textit{a posteriori} error estimation is available for ODEs as well as certain classes of partial differential equations (PDEs) \cite{ODE_Eriksson_Estep_Hansbo_Johnson_1995,petzolode,ainsworth1997posteriori,bangerth2003adaptive,barth2004posteriori,becker2001optimal,carey2009posteriori,chaudhry2015posteriori,chaudhry2021posteriori,chaudhry2018posteriori,chaudhry2017posteriori,chaudhry2015posterioriIMEX,chaudhry2016posteriori,chaudhry2013posteriori,chaudhry2019posteriori,collins2015posteriori,collins2014posteriori,estep1995posteriori,estep2009error,estep2012posteriori,estep2002accounting,estep2000estimating,johansson2015adaptive,chaudhry2021error}, no such analysis has been carried out for DAEs to the best of our knowledge.  However, DAEs have been analyzed in other contexts of error estimation and adjoint analysis. For example, error estimation for  collocation solutions to  linear index-1 DAEs has been studied in \cite{defectauzinger2007defect,defectauzinger2009defect}, while  local error control has been investigated in \cite{errorlocalcontrolsieber2005local}. Various \textit{a priori} error analyses have been carried out in  \cite{errorpetzolhindmarsh1995algorithms,errorRKhairer1988error,errorIMEXboscarino2007error,errorQRlinh2011qr,errornsdaeindex2arnold1995errors,errorDAEOSTIGOVosti_5678419,erroriterateddefectkarasozen1996error,marz2000unified}.
    The adjoint system for index-1 DAEs has been analyzed in  \cite{balla2000linear}. 
    Adjoint based sensitivity analysis for DAEs has been carried out in \cite{adjointsensitivitypetzoldae}.
    
    


    In this article we carry out a thorough adjoint based error analysis for semi-explicit index-1 and Hessenberg index-2 DAEs. In particular, this involves defining the associated adjoint DAE with consistent initialization depending on the structure of the QoI,  and using the adjoint solution to quantify the error in  the computed QoI. The theoretical results in this article derive accurate error estimates for various QoIs ( e.g. depending on final time, average over time intervals, involving either differential or algebraic variables, or both). 
    We develop two error analyses in this article. The first approach is based on defining the adjoint directly to the DAE system, and requires careful setting of the adjoint initial conditions. The second approach defines an adjoint problem to the ODE corresponding to the DAE, which in turn allows the use of classical error analyses techniques. The accuracy of the resulting error estimates is demonstrated through a number of numerical    examples including nonlinear, non-autonomous DAEs as well as a semi-discretized PDE.

    The remainder of this paper is organized as follows: In \cref{sec:background} we discuss  two special types of DAEs;  semi-explicit index-1 and Hessenberg index-2 DAEs. Further we  discuss their numerical treatment and also introduce quantities of interest in this section. In \cref{sec:adj_error_analysis_sec_adj_dae,sec:adj_error_analysis_sec_adj_ode} we develop two distinct analyses to quantify the error in the QoI for DAEs of type semi-explicit index-1 and Hessenberg index-2. The first technique, called Adjoint DAE, is  based on forming an adjoint to the DAE system and is presented in \cref{sec:adj_error_analysis_sec_adj_dae}. The second technique, called Adjoint ODE and discussed in \cref{sec:adj_error_analysis_sec_adj_ode}, first forms an ODE system corresponding to the DAE and then utilizes the adjoint to the ODE system. In \cref{sec:err_estimates} we discuss implementation aspects and formation of error estimates.    
    We show numerical results  for several examples, including nonlinear, non-autonomous DAEs, and a system of partial differential-algebraic equations,
    in \cref{sec:numerical_results}.
    In \cref{sec:conclusion} we give a summary of our contributions with potential directions for future research.


\section{Quantities of Interest in Differential Algebraic Equations}\label{sec:background}

This sections gives a brief background on  semi-explicit index-1 and Hessenberg index-2 Differential-Algebraic Equations (DAEs), their numerical approximation, and introduces two quantities of interest (QoIs) that are the focus of error estimation of this article.

\subsection{Semi-explicit DAEs}
Consider the semi-explicit DAE, 
\begin{subequations}\label{eq:dae_equation}
\begin{align}    
    \dot{y} &= f(y, z, t), \label{eq:dae_diff}\\
    0 &= g(y, z, t) \label{eq:dae_alg},    
\end{align}
\end{subequations}
 where $y(t) \in \mathbb{R}^n$ and $z(t) \in \mathbb{R}^m$ are called the differential and algebraic variables respectively. This form of a DAE, which is an ODE with constraints, arises in numerous science and engineering applications~\cite{nsivpdaepetzol}. Clearly, the initial conditions $y_0 = y(0),z_0=z(0)$ must satisfy the condition that 
\begin{equation}
\label{eq:DAE_IC_I}
 g(y_0,z_0,0) = 0.
 \end{equation}
That is to say, that the initial conditions must satisfy the algebraic constraint in order to be ``consistent.''

The index of a DAE is an important property of the problem that plays a role in classification and has many consequences for the behavior of solutions \cite{nsivpdaepetzol}. There are, in fact, several distinct, but related, definitions of the index of a DAE. 
Here we restrict ourselves to the discussion of the ``differential index.''
For a DAE of the form given by \cref{eq:dae_equation}, the differential index  is the minimum number of times the algebraic  constraint equations needs to be differentiated in order to obtain differential equations for all of the algebraic variables~\cite{Atkinson}. Throughout this article we assume that the DAE~\cref{eq:dae_equation} is either index-1 or index-2.

\subsubsection{Semi-explicit index-1 DAE}\label{subsec:index_1_subsec}
Consider a semi-explicit DAE of the form \cref{eq:dae_equation}. Differentiating the algebraic constraint \cref{eq:dae_alg} with respect to $t$ and formally solving for $\dot{z}$ gives
\begin{align}\label{eq:index_1_zdot}
    \dot{z} &= -[g_{z}]^{-1}\left [g_{y}f + g_{t}\right],
\end{align}
where $g_y = \partial g / \partial y  \in \mathbb{R}^{m\times n}$ is the Jacobian matrix of the function $g$ with respect to $y$ and the $ij$-th component of $g_y$ is: $(g_y)_{ij}=\frac{\partial g_i}{\partial y_j}$ for $i=1,2,\dots,m$, and $j=1,2,\dots,n$. Similarly, $g_z\in \mathbb{R}^{m\times m}$ is the Jacobian of $g$ with respect to $z$ and $g_t=\frac{\partial g}{\partial t}$.
It follows from \cref{eq:index_1_zdot}   that the invertibility of the Jacobian matrix $g_{z}$ in a neighborhood of the solution of \cref{eq:dae_equation} is required to yield explicit differential equations corresponding to the algebraic variables $z(t)$ \cite{solvingodestiffanddae}. 
Indeed, the DAE \cref{eq:dae_equation} is index-1 if and only if $g_{z}$ is invertible \cite{nsivpdaepetzol}.

\subsubsection{Hessenberg (Pure) index-2 DAE}\label{subsec:hessenberg_subsec}
We now consider a common class of DAEs for which the index is greater than one. 
Consider the semi-explicit DAE in case that the algebraic constraint \cref{eq:dae_alg} does not explicitly depend on the algebraic variable $z$. In this case, the DAE may be written,
\begin{subequations}\label{eq:index_2_dae_eq}
\begin{align}
    \dot{y} &= f(y, z, t), \label{eq:index_2_dae_diff}\\
    0 &= g(y, t). \label{eq:index_2_dae_alg}
\end{align}
\end{subequations}
Taking the  derivative of \cref{eq:index_2_dae_alg} with respect to $t$ yields the  so-called hidden constraint,
    \begin{align}\label{eq:hidden_constraint}
        0 = g_yf + g_t.
    \end{align}
Now, in the event that $g_yf_z\in \mathbb{R}^{m\times m}$ is invertible, the ODE \cref{eq:index_2_dae_diff} together with the hidden constraint \cref{eq:hidden_constraint} form a index-1 DAE \cite{solvingodestiffanddae}.
To see this, we differentiate  the hidden constraint \cref{eq:hidden_constraint} with respect to time and formally solve for $\dot z$ to yield the differential equations corresponding to the algebraic variables,
    
    \begin{align}\label{eq:zdot_for_index_2}
        \dot{z} = -\left(g_yf_z\right)^{-1}\displaystyle\left(f^Tg_{yy}f +  g_yf_yf + 2g_{yt}f + g_yf_t + g_{tt}\right),
    \end{align}
here, $g_{yy}$ is the derivative of the Jacobian matrix $g_y$ with respect to $y$, a third order tensor of dimension $m\times n\times n$ with components given by
    \begin{align*}
     \left(g_{yy} \right) _{ijk} = \frac{\partial^2g_i}{\partial y_j\partial y_k}, \quad \text{ for } i=1,2,\dots,m, \, j=1,2,\dots,n, \text{ and } k=1,2,\dots,n.
    \end{align*}
For a fixed $i$, each slice of this tensor represents a Hessian matrix. 
Because $f\in \mathbb{R}^n$, $f^Tg_{yy}f$ has components
\[
\left( f^T g_{yy} f \right)_i = \sum_{j=1}^{n} \sum_{k=1}^{n} \frac{\partial^2 g_i}{\partial y_j \partial y_k} f_j f_k,\quad \text{ for }  i=1,2,\dots, m.
\]
The DAE \cref{eq:index_2_dae_eq} with $g_yf_z$ nonsingular in a neighborhood of the solution is called Hessenberg (Pure) index-2 DAE \cite{nsivpdaepetzol}. 
This is called a pure index-2 because all the algebraic variables (components of $z$) are of index-2. 
The initial conditions $(y_0 = y(0),z_0=z(0)$ must now satisfy
\begin{equation} 
\label{eq:DAE_IC_II}
0 = g(y_0,0), \quad \text{ and } \quad 0 = g_y(y_0,0)f(y_0,z_0,0)+g_t(y_0,0),
\end{equation}
in order to be consistent \cite{solvingodestiffanddae}.

\subsection{Numerical Solution of DAEs}\label{subsec:ns}
There are two main ways to solve DAEs numerically. 
The first is to directly discretize the DAE as written. 
The second is to reformulate the DAE as an ODE through a process known as index reduction, and then numerically solve that ODE.
Direct discretization is often preferred due to the cost of index reduction and to preserve the constraint \cref{eq:dae_alg} exactly for the numerical solution.
The well studied class of numerical methods for ODEs called backward differentiation formula (BDF) methods were applied first in 1971 by Gear \cite{Gear:1970pw} to solve DAEs numerically. 
These methods serve as the basis for the DASSL code \cite{nsivpdaepetzol, DASSL}. 
It is known that the $p$-step BDF method is accurate to order $p$ for $p\leq6$ and effective for solving index-1 as well as Hessenberg index-2 DAEs numerically \cite{Atkinson,nsivpdaepetzol}.

We denote the approximate solution to \cref{eq:dae_equation} or \cref{eq:index_2_dae_eq} as $[Y(t),Z(t)]^T$,
where $Y(t) \in \mathbb{R}^{n}$ and $Z(t) \in \mathbb{R}^{m}$.
In this work, we use the simplest first order BDF method, also known as the Implicit Euler Method, for the numerical solution of index-1 \cref{eq:dae_equation} as well as Hessenberg index-2 DAEs \cref{eq:index_2_dae_eq}. This method is first order accurate, stable and convergent for semi-explicit index-1 and Hessenberg index-2 DAEs \cite{Atkinson}. 
The approximate solution is calculated at a discrete set of nodes,
\begin{align}
    0=t_0<t_1<t_2<\dots <t_N= T.
\end{align}
We take these nodes to be evenly spaced, $t_k=t_0+k\Delta t$, $k=0,1,2,\dots ,N$, where $\Delta t = T/N$. 
Let  $Y_k=Y(t_k)$, and $Z_k=Z(t_k)$ for $k=0,1,\dots,N$, and set $Y_0=y(0), \, Z_0 = z(0)$.  
Applying the BDF-1 method  to the DAE \cref{eq:dae_equation} yields,
\begin{equation}\label{eq:bdf1}
 \left.
    \begin{aligned}
       Y_{k+1} &= Y_k + \Delta t f(Y_{k+1},Z_{k+1},t_{k+1}),\\
    0 &= g(Y_{k+1},Z_{k+1},t_{k+1}),
    \end{aligned}\right\}
\end{equation}
for $k=0,1,\dots,N-1$.
In the case of Hessenberg (Pure) index-2 DAE \cref{eq:index_2_dae_eq}, the constraint equation does not explicitly depend on the algebraic variable and the first order BDF method becomes
\begin{equation}\label{eq:bdf1_for_index2}
    \left.
    \begin{aligned}
        Y_{k+1} &= Y_k + \Delta t f(Y_{k+1},Z_{k+1},t_{k+1}),\\
    0 &= g(Y_{k+1},t_{k+1}),
    \end{aligned} \right \}
\end{equation}
for $k=0,1,\dots,N-1$.
This discretization is implicit in time and yields a system of $n+m$ equations in $n+m$ unknowns. 
Given values for $Y_{k}$, and $Z_{k}$, one must solve this nonlinear system for $Y_{k+1}$, and $Z_{k+1}$.
The numerically computed solution to the index-1 DAEs \cref{eq:dae_equation} and index-2 DAEs \cref{eq:index_2_dae_eq} are naturally calculated at the discrete time points $t_k$. 
In this work, whenever we must evaluate the numerical solution any \emph{other} time point $ t\in (0,T]$, we do so using linear interpolation between the two closet time points such that $t \in (t_k,t_{k+1})$.

\subsection{Quantity of Interest (QoI)}\label{subsec:qoi}
In this section we formalize the notion of a Quantity of Interest or QoI. 
In many applications, the approximate solution of an DAE is used to calculate some measurable quantity, called the QoI.
In this work we model QoIs as measurable quantities that can be expressed as a linear functional acting on the solution of the DAE. To describe the QoIs considered, we first introduce some notation. For any two time-dependent functions $a,b \in [\mathcal{L}^2(0,T]]^{d}$, we define
\begin{align}\label{eq:definition_integration}
    \langle a,b\rangle=\int_{0}^{T}\left(a(t),b(t)\right)\,\,dt,
\end{align}
where $(\cdot,\cdot)$ is the usual Euclidean inner-product in $\mathbb{R}^d$.  


We consider two types of QoI in this work. The first QoI is defined over the  interval $\left[0,T\right]$ as,
\begin{align}\label{eq:qoi_def_inverval}
    \QCI{y}{z} = \langle y, \psi^y \rangle + \langle z, \psi^z \rangle,
\end{align}
where $\psi^y \in [\mathcal{L}^2(0,T]]^{n}$ and  $\psi^z \in [\mathcal{L}^2(0,T]]^{m}$.
The other type of QoI depends only on the solution at the terminal time $T$, and is defined as,
\begin{align}\label{eq:qoi_def_at_terminal_time}
    \QFT{y}{z} = \left(y(T), \zeta^y\right) + \left(z(T), \zeta^z\right),
\end{align}
where $\zeta^y\in \mathbb{R}^{n}$ and $\zeta^z\in \mathbb{R}^{m}$.

\subsection{Error in QoI}\label{subsec:error_in_qoi}
Now, given an approximate solution to a DAE, our main goal is to calculate
the error in the computed value of the QoIs.
Recalling that $[y,z]^T$ is the \emph{true} solution of \cref{eq:dae_equation} and $[Y,Z]^T$ the numerically computed solution, let $e^{(y)} = (y - Y)$ and $e^{(z)} = (z - Z)$.
The errors in the two QoIs, based on the definition in \cref{eq:qoi_def_inverval,eq:qoi_def_at_terminal_time},  are given below.
\begin{itemize}
    \item The error in the QoI $\QCI{Y}{Z}$ is
    \begin{align}\label{eq:qoi_over_time}
        \eQCI := \QCI{y}{z} - \QCI{Y}{Z} &= \langle e^{(y)} , \psi^y\rangle +  \langle e^{(z)} , \psi^z\rangle.
    \end{align}
    
   
    \item Similarly, the error in the QoI $\QFT{Y}{Z}$ is
    \begin{align}\label{eq:qoi_terminal_time_error}
       \eQFT := \QFT{y}{z}- \QFT{Y}{Z} = \left( e^{(y)}(T), \zeta^y\right) + \left( e^{(z)}(T), \zeta^z\right).
    \end{align}
\end{itemize}

Quantifying these errors is the focus of the next two sections. In \cref{sec:adj_error_analysis_sec_adj_dae}, the analysis is carried out by defining an adjoint problem to the DAE system~\cref{eq:dae_equation}. Later, in \cref{sec:adj_error_analysis_sec_adj_ode}, index reduction is used to convert the DAE to an ODE, and then classical analysis is employed to derive error estimates.

\section{Error Analysis of DAEs by Adjoint to the DAE system}\label{sec:adj_error_analysis_sec_adj_dae}
The error analysis in this section, called Adjoint DAE analysis, is carried out by defining an adjoint problem to the DAE system~\cref{eq:dae_equation}, and by a careful choice of the initial conditions for the adjoint problem.

\subsection{Adjoint System}\label{subsec:adjoint_dae_general}
We introduce adjoint differential variables $\phi^{(y)}(t)\in\mathbb{R}^n$ and adjoint algebraic variables $\phi^{(z)}(t)\in\mathbb{R}^m$, and define the corresponding adjoint problem for a system of semi-explicit DAE \cref{eq:dae_equation} as follows,
\begin{subequations}\label{eq:adj_all_time}
    \begin{align}
    -\dot{\phi}^{(y)} &= \bar{f}_y^T \phi^{(y)} + \bar{g}_y^T \phi^{(z)} + \psi^y,\quad t\in [0,T), \label{eq:diff}\\
0 &= \bar{f}_z^T \phi^{(y)} + \bar{g}_z^T \phi^{(z)} + \psi^z,\quad t\in [0,T),\label{eq:alg}
    \end{align}
\end{subequations}
where the linearized operators are
$$\displaystyle \bar{f}_y = \int_0^1 \frac{\partial f(\tilde{y},\tilde{z})}{ \partial y}ds \in \mathbb{R}^{n\times n},\quad\quad \displaystyle \bar{f}_z = \int_0^1 \frac{\partial f(\tilde{y},\tilde{z})}{ \partial z}ds \in \mathbb{R}^{n\times m},$$
$$\displaystyle \bar{g}_y = \int_0^1 \frac{\partial g(\tilde{y},\tilde{z})}{ \partial y}ds \in \mathbb{R}^{m\times n},\quad\quad \displaystyle \bar{g}_z = \int_0^1 \frac{\partial g(\tilde{y},\tilde{z})}{ \partial z}ds\in \mathbb{R}^{m\times m},$$
for $\tilde{y}=sy + (1-s)Y$, and $\tilde{z}=sz + (1-s)Z$.
The consistent initial conditions corresponding to adjoint algebraic variables $\phi^{(z)}$ at $t=T$ are,
\begin{equation}
\phi^{(z)}(T) = 
\begin{cases}
-\left[\bar{g}_z^T\right]^{-1}\left[\bar{f}_z^T\phi^{(y)}+\psi^z\right]\Bigr|_{t=T}, & \text{if adjoint DAE is index-1},\\
-\left(\bar{f}_z^T\bar{g}_y^T\right)^{-1}\left[\bar{f}_z^T\bar{f}_y^T\phi^{(y)}+\bar{f}_z^T\psi^y -\psi_t^z\right]\Bigr|_{t=T},  & \text{if adjoint DAE is index-2}.
\end{cases}
\label{eq:phi_z_final}
\end{equation}
These conditions follow directly from \cref{eq:DAE_IC_I,eq:DAE_IC_II}.
Here, we leave the initial conditions $\phi^{(y)}(T)$  unspecified for the moment. 

\begin{theorem}
    The adjoint DAE system \cref{eq:adj_all_time} preserves the index structure of the the original DAE system \cref{eq:dae_equation} provided the numerical solution $[Y,Z]^T$ is sufficiently close to the true solution $[y,z]^T$. That is, the adjoint DAE system has the same index as the original DAE provided $e^{(y)}$ and $e^{(z)}$ are sufficiently small.
\end{theorem}

\begin{proof}
First assume the DAE \cref{eq:dae_equation} is index-1. Then, from the discussion in \cref{subsec:index_1_subsec}, the adjoint DAE \cref{eq:adj_all_time} is index-1 if and only if $\bar{g}_z^T$, or equivalently, $\bar{g}_z$ is invertible. Now, since the original DAE is index-1, $g_z$ is invertible. However, this does not directly imply that $\bar{g}_z$ is also invertible. Note that $\bar{g}_z$   may be written as,
\[
\bar{g}_z  =  \int_0^1 g_z (se^{(y)} + Y,se^{(z)} + Z)ds.
\]
Consider the following  function of $[u,v]^T \in \mathbb{R}^{n+m}$,
\[
G(u,v) =  \mathrm{det}\left[ \int_0^1 g_z(su + Y, sv + Z)ds\right],
\]
where $\mathrm{det}$ denotes the determinant. Now $G(u,v)$ is a continuous function of $u$ and $v$ since the determinant is a continuous function of its entries.  Further, $G(0,0) =  \mathrm{det}[g_z(Y,Z)] \neq 0$ since $g_z$ is invertible. Without loss of generality assume that $G(0,0) > 0$. Then from the continuity of $G$, there is a ball of radius $\epsilon$ around $[0,0]^T$, $B_\epsilon([0,0]^T)$, such that $G(u,v) > 0$ for $[u,v]^T \in B_\epsilon([0,0]^T)$. Noting that $G(e^{(y},e^{(z)}) = \mathrm{det}[\bar{g}_z]$, we conclude that $\bar{g}_z$ is invertible for $e^{(y)}$ and $e^{(z)}$ sufficiently small.

A similar argument shows that the adjoint DAE is index-2 if the original DAE is index-2 and $e^{(y)}$ and $e^{(z)}$ are sufficiently small.

\end{proof}

Defining the correct initial conditions for $\phi^{(y)}(T)$ is a key component of analysis, and is motivated from the work on adjoint based sensitivity analysis in \cite{adjointsensitivitypetzoldae}. For the adjoint DAE \cref{eq:adj_all_time}, the initial conditions must be set in a manner determined by the index of the original DAE as well as the dependency of QoIs over the interval $[0,T]$, and at the terminal time $T$, as we show in later sections. For now, we note some properties of the linearized operators used later in the analysis. By the fundamental theorem of calculus,
    \begin{align*}
    f(y,z)-f(Y,Z)&=\int_{0}^{1}\frac{d}{d s}\left[f(sy+(1-s)Y,sz+(1-s)Z)\right]\,ds,\\
    &=\int_{0}^{1}\frac{\partial f}{\partial y}ds\,(y-Y) + \int_{0}^{1}\frac{\partial f}{\partial z}ds\,(z-Z),\\
    &=\bar{f}_y(y-Y)+\bar{f}_z(z-Z),
    \end{align*}
so, we have
\begin{align}\label{eq:property_adj_1_property_1}
    \bar{f}_y(y-Y)+\bar{f}_z(z-Z)=f(y,z)-f(Y,Z),
\end{align}
similarly,
\begin{align}\label{eq:property_adj_1_property_2}
    \bar{g}_y(y-Y)+\bar{g}_z(z-Z)&=g(y,z)-g(Y,Z).
\end{align}

The following lemma is used in deriving error representations later.
\begin{lemma}\label{lem:adj_error}
We have 
\begin{align}\label{eq:error_dae1_differential}
   \eQCI &= \left( \phi^{(y)}(0),e^{(y)}(0) \right) - \left(\phi^{(y)}(T),y(T)-Y(T)\right) +\langle \phi^{(y)}, f(Y,Z) - \dot{Y} \rangle+ \langle \phi^{(z)}, g(Y,Z)  \rangle.
\end{align}
\end{lemma}
The proof is quite similar to standard proofs for error analysis of ODEs, however, for completeness we include it in appendix \ref{sec:app_proofs}.

\subsection{Error Analysis for Semi-explicit Index-1 DAE}\label{subsec:error_analysis_index_1_sec}
Now we derive error representations for QoIs computed from the numerical solution of index-1 DAEs.
The error in semi-explicit index 1 DAE for the computed QoI $\QCI{Y}{Z}$ is analyzed in \cref{thm:adj_1_qoi_diff_vari_all_time}, while the error in the computed QoI $\QFT{Y}{Z}$ is analyzed in \cref{thm:adj1_qoi_terminal_time_diff_algeb_vari}. 

\begin{theorem}\label{thm:adj_1_qoi_diff_vari_all_time}
Let $\phi^{(y)}(T)=0$ in the adjoint problem \cref{eq:adj_all_time}.
Then the error in the computed QoI $\QCI{Y}{Z}$ is,
\begin{align}\label{eq:error_dae1_both_diff_alg}
    \eQCI = \left( \phi^{(y)}(0),e^{(y)}(0) \right) + \langle \phi^{(y)}, f(Y,Z) - \dot{Y}  \rangle + \langle \phi^{(z)}, g(Y,Z)  \rangle.
\end{align}
\end{theorem}
\begin{proof}
    The proof directly follows from the \cref{lem:adj_error} by setting  $\phi^{(y)}(T) = 0$.
  
\end{proof}


\begin{theorem}\label{thm:adj1_qoi_terminal_time_diff_algeb_vari}
 Let $\psi^y(t)=0$, and $\psi^z(t)=0$ in \cref{eq:adj_all_time}, and the adjoint initial condition given by 
\begin{align}\label{eq:bcs_index_1_q_on_y_z}
    \phi^{(y)}(T) = \zeta^{y} - \bar{g}^T_{y}\left(\bar{g}^T_{z}\right)^{-1}\Bigr|_{t=T}\zeta^{z}.
\end{align}
Then the error in the computed QoI $\QFT{Y}{Z}$, is given by
\begin{align}\label{eq:error_dae1_differ_algeb_terminal_time}
    \eQFT 
    &= \left( \phi^{(y)}(0),e^{(y)}(0) \right) + \langle \phi^{(y)}, f(Y,Z) - \dot{Y} \rangle + \langle \phi^{(z)}, g(Y,Z)  \rangle.
\end{align}
\end{theorem}

\begin{proof}
Substituting $\psi^y(t)=0$, and $\psi^z(t)=0$ in \cref{lem:adj_error}, we have

    \begin{align}\label{eq:plug_bc_for_index1_phi_T}
        &\left(\phi^{(y)}(T),y(T)- Y(T)\right) 
        =  \langle \phi^{(y)}, f(Y,Z) - \dot{Y}\rangle + \langle \phi^{(z)},g(Y,Z)\rangle + \left(\phi^{(y)}(0),y(0)- Y(0)\right).
    \end{align}
    Considering the left hand side of  \cref{eq:plug_bc_for_index1_phi_T} and using \cref{eq:bcs_index_1_q_on_y_z} and \cref{eq:property_adj_1_property_2},
    \begin{align*}
        \left(\phi^{(y)}(T),e^{(y)}(T)\right)  
        &= \left(\zeta^{y},e^{(y)}(T)\right) - \left(\bar{g}^T_{y}\left(\bar{g}^T_{z}\right)^{-1}\zeta^{z}, y-Y\right)\Bigr|_{t=T},\\
        &= \left(\zeta^{y},e^{(y)}(T)\right) - \left(\left(\bar{g}^T_{z}\right)^{-1}\zeta^{z}, \bar{g}_{y}\left(y - Y\right)\right)\Bigr|_{t=T},\\
        &= \left(\zeta^{y},e^{(y)}(T)\right) - \left(\left(\bar{g}^T_{z}\right)^{-1}\zeta^{z}, g(y,z) - g(Y,Z) - \bar{g}_{z}(z-Z)\right)\Bigr|_{t=T},\\
        &= \left(\zeta^{y},e^{(y)}(T)\right) + \left(\left(\bar{g}^T_{z}\right)^{-1}\zeta^{z}, g(Y,Z)\right)\Bigr|_{t=T} + \left(\left(\bar{g}^T_{z}\right)^{-1}\zeta^{z},\bar{g}_{z}(z-Z)\right)\Bigr|_{t=T},\\
        &= \left(\zeta^{y},e^{(y)}(T)\right) + \left(\zeta^{z},e^{(z)}(T)\right),\\
     \end{align*}
    where we used  $g(Y, Z)|_{t=T=t_N} =  0$ in the last step.
    Combining this with \cref{eq:plug_bc_for_index1_phi_T} and noticing that 
    $\eQFT = \left(\zeta^{y},e^{(y)}(T)\right) + \left(\zeta^{z},e^{(z)}(T)\right)$
    proves the theorem.
    
\end{proof}


\begin{remark}
Quite often the QoI $\QFT{y}{z}$ depends on only $y$, i.e.,  $\QFT{y}{z} = \left(y(T), \zeta^y\right)$. In this case we have $\zeta^{z} = 0$, and the initial condition simplifies to $\phi^{(y)}(T)=\zeta^y$, while the error representation remains \eqref{eq:error_dae1_differ_algeb_terminal_time}.
\end{remark}



\subsection{Error Analysis for Hessenberg Index-2 DAE}\label{subsec:error_analysis_Hessenberg_sec}

The error in Hessenberg index 2 DAE for the computed QoI $\QCI{Y}{Z}$ is analyzed in \cref{thm:error_dae_index2_updated}, while the error in the computed QoI $\QFT{Y}{Z}$ is analyzed in \cref{thm:adj2_qoi_terminal_time_both_diff_alg_vari}. 


\begin{theorem}\label{thm:error_dae_index2_updated}
Consider \cref{eq:adj_all_time} and let $\phi^{(y)}(T)$ be given by 
\begin{align}\label{eq:phi_1(T)_final}
    \phi^{(y)}(T)=-\bar{g}_y^T\left(\bar{f}_z^T\bar{g}_y^T\right)^{-1}\psi^z\Bigr|_{t=T}.
\end{align}
Then the error in the computed QoI $\QCI{Y}{Z}$, is given by
\begin{align}\label{eq:error_dae_index2_differential_updated}
    \eQCI &= \left( \phi^{(y)}(0),e^{(y)}(0) \right) + \langle \phi^{(y)}, f(Y,Z) - \dot{Y}  \rangle+ \langle \phi^{(z)}, g(Y)  \rangle.
\end{align}

\end{theorem}
\begin{proof}
    Since $g$ is independent of $z$, we have  $\bar{g}_z=0$. Using this in \cref{lem:adj_error} leads to,
    \begin{align}\label{eq:proof_of_error_in_index_2_cumulative}
     \eQCI &= \left( \phi^{(y)}(0),e^{(y)}(0) \right) - \left(\phi^{(y)}(T),y(T)-Y(T)\right) + \langle \phi^{(y)}, f(Y,Z) - \dot{Y} \rangle + \langle \phi^{(z)}, g(Y)  \rangle.
    \end{align}
Clearly, the term $
    \left(\phi^{(y)}(T),y(T)-Y(T)\right)$ in \cref{eq:proof_of_error_in_index_2_cumulative} is not computable because of the presence of unknown $y(T)$. Using \cref{eq:phi_1(T)_final} in this  term, followed by \cref{eq:property_adj_1_property_2} leads to,
\begin{align}\label{eq:expression}
     \left(\phi^{(y)}(T),y(T)-Y(T)\right) &= \left( -\bar{g}_y^T\left(\bar{f}_z^T\bar{g}_y^T\right)^{-1}\psi^z, e^{(y)}\right)\Bigr|_{t=T} ,\notag\\
    &= - \left( \left(\bar{f}_z^T\bar{g}_y^T\right)^{-1}\psi^z, \bar{g}_y(y-Y)\right)\Bigr|_{t=T} ,\notag\\
    &= - \left(\left(\bar{f}_z^T\bar{g}_y^T\right)^{-1}\psi^z,  g(y)-g(Y)\right)\Bigr|_{t=T} ,\notag\\
    &= 0,
\end{align}
where we used $g(y) = 0$, and $g(Y)|_{t=T=t_N} =  0$ in the last step. Combining \cref{eq:expression} and \cref{eq:proof_of_error_in_index_2_cumulative} proves the result.

\end{proof}


Next we prove some technical results which are utilized in the analysis for the computed QoI $\QFT{Y}{Z}$. To this end we define,
\begin{equation}\label{eq:def_P}
P = I - \bar{f}_{z}\left(\bar{g}_{y}\bar{f}_{z}\right)^{-1}\bar{g}_{y}. 
\end{equation}

\begin{lemma}\label{lem:tech_results}
    We have,
    \begin{enumerate}
        \item $\left(P^T\zeta^y,e^{(y)}\right) = \left((\bar{f}^T_{z}\bar{g}^T_{y})^{-1}\bar{f}_z^T\zeta^y,g(Y)\right)$,
        \item $\displaystyle \left(-P^T\bar{f}^T_{y} \bar{g}^T_{y}(\bar{f}^T_{z}\bar{g}^T_{y})^{-1}\zeta^z , e^{(y)}\right)$ $=$\\
        $\displaystyle \phantom{a}\quad -\left(\bar{g}^T_{y}(\bar{f}^T_{z}\bar{g}^T_{y})^{-1}\zeta^z ,f(y,z)\right) + \left(\bar{g}^T_{y}(\bar{f}^T_{z}\bar{g}^T_{y})^{-1}\zeta^z ,f(Y,Z)\right) +\left(\zeta^z ,e^{(z)}\right)$\\
        $\displaystyle \phantom{a}\quad - \left((\bar{f}^T_{z}\bar{g}^T_{y})^{-1}\bar{f}_z^T\bar{f}^T_{y}\bar{g}^T_{y}(\bar{f}^T_{z}\bar{g}^T_{y})^{-1}\zeta^z,g(Y)\right)$,
        \item $\displaystyle \left( \frac{d \bar{g}^T_{y}}{dt}\left(\bar{f}^T_{z}\bar{g}^T_{y}\right)^{-1} \zeta^z,e^{(y)}\right) = $\\
 $\displaystyle \phantom{a}\quad -\left(\left(\bar{f}^T_{z}\bar{g}^T_{y}\right)^{-1} \zeta^z,\frac{d g(Y)}{dt}\right) - \left(\bar{g}^T_{y}\left(\bar{f}^T_{z}\bar{g}^T_{y}\right)^{-1} \zeta^z,f(y,z)\right)+ \left(\bar{g}^T_{y}\left(\bar{f}^T_{z}\bar{g}^T_{y}\right)^{-1} \zeta^z,\dot{Y}\right)$,
    \end{enumerate}
    where all the functions are evaluated at a fixed but arbitrary $t \in [0,T]$.
\end{lemma}
The proof of the lemma is given in appendix \ref{sec:app_proofs}.

\begin{theorem}\label{thm:adj2_qoi_terminal_time_both_diff_alg_vari}
Consider \cref{eq:adj_all_time} with $\psi^y(t)=0,\,\psi^z(t)=0$, and the initial condition for $\phi^{(y)}(T)$ given by 
\begin{align}\label{eq:phi_T_at_terminal_when_fzgy_not_constant}
    \phi^{(y)}(T) &= P^T\left[ \zeta^y - \bar{f}^T_{y} \bar{g}^T_{y}(\bar{f}^T_{z}\bar{g}^T_{y})^{-1}\zeta^z  - \frac{d \bar{g}^T_{y}}{dt}\left(\bar{f}^T_{z}\bar{g}^T_{y}\right)^{-1} \zeta^z \right]_{t=T}
\end{align}
 
Then the error in the computed QoI $\QFT{Y}{X}$, is given by
\begin{align}\label{eq:error_dae2_both_dif_alg_terminal_time}
    \eQFT
    &= \left(\phi^{(y)}(0),e^{(y)}(0)\right) + \langle \phi^{(y)}, f(Y,Z) - \dot{Y}\rangle + \langle \phi^{(z)},g(Y)\rangle\notag\\
    &\quad\quad - \left(\bar{g}^T_{y}(\bar{f}^T_{z}\bar{g}^T_{y})^{-1}\zeta^z ,f(Y,Z)-\dot{Y}\right)\Bigr|_{t=T} -\left(\left(\bar{f}^T_{z}\bar{g}^T_{y}\right)^{-1} \zeta^z,\frac{d g(Y)}{dt}\right)\Bigr|_{t=T}.\notag\\
\end{align}
\end{theorem}
\begin{proof}
    Starting from  equation \cref{eq:plug_bc_for_index1_phi_T} we have
    \begin{align}\label{eq:plug_bc_for_index2_phi_T}
        \left(\phi^{(y)}(T),y(T)- Y(T)\right) &= \left(\phi^{(y)}(0),e^{(y)}(0)\right) + \langle \phi^{(y)}, f(Y,Z) - \dot{Y}\rangle + \langle \phi^{(z)},g(Y)\rangle.
    \end{align}
    Focusing on the left hand side of \cref{eq:plug_bc_for_index2_phi_T} by combining it with \cref{eq:phi_T_at_terminal_when_fzgy_not_constant} we arrive at,
    \begin{align}\label{eq:plug_k's}
        &\left(\phi^{(y)}(T),e^{(y)}(T)\right)\notag\\
        &= \left(P^T\left[ \zeta^y - \bar{f}^T_{y} \bar{g}^T_{y}(\bar{f}^T_{z}\bar{g}^T_{y})^{-1}\zeta^z  - \frac{d \bar{g}^T_{y}}{dt}\left(\bar{f}^T_{z}\bar{g}^T_{y}\right)^{-1} \zeta^z \right],e^{(y)}\right)\Bigr|_{t=T}\notag,\\
        &= \left(P^T\zeta^y,e^{(y)}(T)\right)\Bigr|_{t=T} + \left(-P^T\bar{f}^T_{y} \bar{g}^T_{y}(\bar{f}^T_{z}\bar{g}^T_{y})^{-1}\zeta^z , e^{(y)}(T)\right)\Bigr|_{t=T}\notag\\
        &\quad\quad\quad\quad - \left( \frac{d \bar{g}^T_{y}}{dt}\left(\bar{f}^T_{z}\bar{g}^T_{y}\right)^{-1} \zeta^z,e^{(y)}(T)\right)\Bigr|_{t=T}
    \end{align}
    Applying \cref{lem:tech_results}  to the terms in \cref{eq:plug_k's}, 
    \begin{align}
        &\left(\phi^{(y)}(T),e^{(y)}(T)\right)\notag\\
        &= \left(\zeta^y,e^{(y)}(T)\right) + \left((\bar{f}^T_{z}\bar{g}^T_{y})^{-1}\bar{f}_z^T\zeta^y,g(Y)\right)\Bigr|_{t=T} -\left(\bar{g}^T_{y}(\bar{f}^T_{z}\bar{g}^T_{y})^{-1}\zeta^z ,f(y,z)\right)\Bigr|_{t=T}\notag\\
        &\quad + \left(\bar{g}^T_{y}(\bar{f}^T_{z}\bar{g}^T_{y})^{-1}\zeta^z ,f(Y,Z)\right)\Bigr|_{t=T} +\left(\zeta^z ,e^{(z)}(T)\right) - \left((\bar{f}^T_{z}\bar{g}^T_{y})^{-1}\bar{f}_z^T\bar{f}^T_{y}\bar{g}^T_{y}(\bar{f}^T_{z}\bar{g}^T_{y})^{-1}\zeta^z,g(Y)\right)\Bigr|_{t=T}\notag\\
        &\quad + \left(\left(\bar{f}^T_{z}\bar{g}^T_{y}\right)^{-1} \zeta^z,\frac{d g(Y)}{dt}\right)\Bigr|_{t=T} + \left(\bar{g}^T_{y}\left(\bar{f}^T_{z}\bar{g}^T_{y}\right)^{-1} \zeta^z,f(y,z)\right)\Bigr|_{t=T} - \left(\bar{g}^T_{y}\left(\bar{f}^T_{z}\bar{g}^T_{y}\right)^{-1} \zeta^z,\dot{Y}\right)\Bigr|_{t=T},\notag\\
        &= \eQFT+ \left(\bar{g}^T_{y}(\bar{f}^T_{z}\bar{g}^T_{y})^{-1}\zeta^z ,f(Y,Z)-\dot{Y}\right)\Bigr|_{t=T} + \left(\left(\bar{f}^T_{z}\bar{g}^T_{y}\right)^{-1} \zeta^z,\frac{d g(Y)}{dt}\right)\Bigr|_{t=T}\notag,\\
    \end{align}
    where we used $g(Y)|_{t=T=t_N} =  0$ in the last step. Combining this  equation  with \cref{eq:plug_bc_for_index2_phi_T} proves the result.
\end{proof}


\begin{remark}
    Quite often the QoI $\QFT{y}{z}$ at $t=T$ depends on only $y$, i.e.,  $\QFT{y}{z} = \left(y(T), \zeta^y\right)$. In this case we have $\zeta^{z} = 0$, and the initial condition simplifies to 
    \begin{equation}\label{eq:adj_index_2_terminal_time_q_on_y_onlhy}
    \phi^{(y)}(T) = \left( I - \bar{g}^T_{y}(\bar{f}^T_{z}\bar{g}^T_{y})^{-1}\bar{f}^T_{z}\right)\zeta^y\Bigr|_{t=T},
    \end{equation}
    while the error representation \eqref{eq:error_dae2_both_dif_alg_terminal_time} takes the form,
    \begin{align}\label{eq:error_dae2_differ_terminal_time}
    \eQFT &= \left( e^{(y)}(0),\phi^{(y)}(0) \right) + \langle \phi^{(y)}, f(Y,Z) - \dot{Y}  \rangle + \langle \phi^{(z)}, g(Y)  \rangle.
\end{align}
\end{remark}



\section{Error Analysis of DAEs by Adjoint to the Index-Reduced ODE system}\label{sec:adj_error_analysis_sec_adj_ode}
A DAE corresponds to an ODE (index-0 DAE) with an invariant \cite{computermethodpetzold}. We take advantage of this  fact to utilize classical error analysis techniques to form error estimates for DAEs. Note that the corresponding ODE system is never solved numerically, and is only an analytical device used to derive error representations. We refer to the resulting error analysis as the Adjoint ODE analysis.

\subsection{Adjoint System}\label{subsec:ode_adjoint}
 For the DAE \cref{eq:dae_equation} the corresponding ODE is  

\begin{subequations}\label{eq:ode_reduced_form}
    \begin{align}
        \dot{y} &= f(y,z,t) \label{eq:ode_diff_vari},\\
        \dot{z} &= h(y,z,t) \label{eq:ode_algebraic}.
    \end{align}
\end{subequations}
with an invariant $g(y,z,t)=0$, and an additional invariant \cref{eq:hidden_constraint} for index-2 DAE.
The function $h(y,z,t)$ is defined by \cref{eq:index_1_zdot} if the DAE is  index-1 and  by \cref{eq:zdot_for_index_2} if the DAE is index-2,
\begin{equation}
h = 
\begin{cases}
-[g_{z}]^{-1}\left [g_{y}f + g_{t}\right], & \text{if DAE is index-1},\\
-\left(g_yf_z\right)^{-1}\displaystyle\left(f^Tg_{yy}f +  g_yf_yf + 2g_{yt}f + g_yf_t + g_{tt}\right),  & \text{if DAE is index-2}.
\end{cases}
\label{eq:h_definition}
\end{equation}
Note that \cref{eq:dae_equation,eq:ode_reduced_form} have the same solution the solution $[y, \, z]^T$.
We define the corresponding adjoint ODE system for \cref{eq:ode_reduced_form}  as follows:
\begin{subequations}\label{eq:adjoint_ode}
    \begin{align}
        -\dot{\nu}^{(y)} &= \bar{f}_y^T \nu^{(y)} + \bar{h}_y^T \nu^{(z)} + \psi^y, \quad t\in [0,T)\label{eq:ode_adj_diff_equation},\\
-\dot{\nu}^{(z)} &= \bar{f}_z^T \nu^{(y)} + \bar{h}_z^T \nu^{(z)} + \psi^z, \quad t\in [0,T)\label{eq:ode_adj_alg_equation},
    \end{align}
\end{subequations}
with the initial condition $\nu^{(y)}(T)$, and $\nu^{(z)}(T)$ unspecified, for the moment.
Here
$$\displaystyle \bar{h}_y = \int_0^1 \frac{\partial h(\tilde{y},\tilde{z})}{ \partial y}ds\in \mathbb{R}^{m\times n},\quad\quad \displaystyle \bar{h}_z = \int_0^1 \frac{\partial h(\tilde{y},\tilde{z})}{ \partial z}ds\in \mathbb{R}^{m\times m},$$
for $\tilde{y}=sy + (1-s)Y$, and $\tilde{z}=sz + (1-s)Z$. 
Similar to  \cref{eq:property_adj_1_property_1} and \cref{eq:property_adj_1_property_2} we have, 
\begin{align}\label{eq:property_odeadj_property_1}
    \bar{h}_y(y-Y)+\bar{h}_z(z-Z)=h(y,z)-h(Y,Z).
\end{align}

\subsection{Error Analysis for Semi-explicit Index-1 and Hessenberg Index-2 DAE}

\begin{lemma}\label{lem:err_using_ode_adj}
    We have,
\begin{align}\label{eq:err_da1_both_al_diff_vari}
    &\eQCI + \eQFT\notag\\
    &= \left( \nu^{(y)}(0),e^{(y)}(0) \right)
    + \left( \nu^{(z)}(0),e^{(z)}(0) \right) +\langle \nu^{(y)}, f(Y,Z) - \dot{Y}  \rangle +\langle \nu^{(z)}, h(Y,Z) - \dot{Z}  \rangle.
\end{align}
\end{lemma}

\begin{proof}
   The proof is standard see \cite{ODE_Eriksson_Estep_Hansbo_Johnson_1995}.
\end{proof}
The Adjoint ODE approach is conceptually simpler than the Adjoint DAE approach in the sense that only a single analysis is required for both Index-1 and Index-2 DAEs. However, we do need slightly different analyses for the two QoIs. These results are given in \cref{thm:err_using_ode_adj_diff} and \cref{thm:terminal_error_index1_from_ode_adj} below.


\begin{theorem}\label{thm:err_using_ode_adj_diff}
 Let $\nu^{(y)}(T)=0$, and $\nu^{(z)}(T)=0$ in the adjoint problem \cref{eq:adjoint_ode}.
Then the error in the computed QoI $\QCI{Y}{Z}$ is,
\begin{align}\label{eq:err_commulative_diff_vari}
    \eQCI
    &= \left( \nu^{(y)}(0),e^{(y)}(0) \right)
    + \left( \nu^{(z)}(0),e^{(z)}(0) \right)
    +\langle \nu^{(y)}, f(Y,Z) - \dot{Y}  \rangle +\langle \nu^{(z)}, h(Y,Z) - \dot{Z}  \rangle.
\end{align}
\end{theorem}

\begin{proof}
The proof directly follows from \cref{lem:err_using_ode_adj} by setting the initial condition $\nu^{(y)}(T)=0$, and $\nu^{(z)}(T)=0$.
\end{proof}


\begin{theorem}\label{thm:terminal_error_index1_from_ode_adj}
Let $\psi^y(t)=0,\, \psi^z(t)=0$ in \cref{eq:adjoint_ode}, and the initial condition is given by $\nu^{(y)}(T)=\zeta^y,\, \nu^{(z)}(T)=\zeta^z$. Then the error in the computed QoI $\QFT{Y}{Z}$, is given by
    \begin{align}\label{eq:err_terminal_diff_vari}
    \eQFT
    &= \left( \nu^{(y)}(0),e^{(y)}(0) \right) + \left( \nu^{(z)}(0),e^{(z)}(0) \right) +\langle \nu^{(y)}, f(Y,Z) - \dot{Y}  \rangle +\langle \nu^{(z)}, h(Y,Z) - \dot{Z}  \rangle.
\end{align}
\end{theorem}
\begin{proof}
    The proof follows directly from \cref{lem:err_using_ode_adj} by substituting  $\psi^y(t)=0$, and $\psi^z(t)=0$.
\end{proof}

\section{Error Estimates}
\label{sec:err_estimates}
The error representations in \cref{thm:adj_1_qoi_diff_vari_all_time,thm:adj1_qoi_terminal_time_diff_algeb_vari,thm:error_dae_index2_updated,thm:adj2_qoi_terminal_time_both_diff_alg_vari} (based on Adjoint DAE) and \cref{thm:err_using_ode_adj_diff,thm:terminal_error_index1_from_ode_adj} (based on Adjoint ODE) are exact, however, they involve the unknown true adjoint solution. Once we replace this adjoint solution by its numerical approximation (as discussed below in \cref{subsec:ns_adj_dae_and_ode}), we obtain an error estimate from the corresponding error representation. 
More precisely, the error estimates based on Adjoint DAE are the right-hand sides of \cref{eq:error_dae1_both_diff_alg,eq:error_dae1_differ_algeb_terminal_time,eq:error_dae_index2_differential_updated,eq:error_dae2_both_dif_alg_terminal_time,eq:error_dae2_differ_terminal_time} 
with  $\phi^{(y)}$ and $\phi^{(z)}$  replaced by their numerical approximations $\overline{\phi}^{(y)}$ and $\overline{\phi}^{(z)}$.
Similarly, error estimates based on the Adjoint ODE are the right-hand sides of \cref{eq:err_commulative_diff_vari,eq:err_terminal_diff_vari}.
Since the error estimate is exactly the same as the error representation, except with the adjoint solution replaced by its numerical approximation, we avoid re-writing the error estimates separately.
However, now our error estimates no longer capture the exact error, so their accuracy must be quantified. 
This aspect is discussed in \cref{subsec:effectivity_ratio}. 
Finally, in \cref{sec:comparison_approaches} we discuss the relative advantages and disadvantages of the error estimates formed from the Adjoint DAE and Adjoint ODE analyses.

\subsection{Numerical Solution of Adjoint DAE and Adjoint ODE}\label{subsec:ns_adj_dae_and_ode}
We now describe the numerical approximation to the adjoint solution in either the DAE \cref{eq:adj_all_time} or ODE \cref{eq:adjoint_ode}. The adjoint equations \cref{eq:adj_all_time,eq:adjoint_ode} technically require the true solution in the linearized operators $\bar{f}_{y},\bar{g}_{y}, \bar{g}_{z}$ and $\bar{h}_{y}$. However, in forming the numerical approximation of the adjoint problem, we substitute the approximate solution $[Y,Z]^T$ for the true solution $[y,z]^T$ in the linearized operators as is standard in adjoint based error estimation~\cite{estep2000estimating,chaudhry2017posteriori,chaudhry2013posteriori,collins2014posteriori}. 
That is, we approximate $\bar{f}_{y}, \bar{f}_{z}, \bar{g}_{y}$, $\bar{g}_{z}$, $\bar{h}_{y}$, and $\bar{h}_{z}$ (see \cref{subsec:adjoint_dae_general,subsec:ode_adjoint} for the definitions of these operators) as follows
$$\displaystyle \bar{f}_{y}\approx \frac{\partial f}{\partial y}\Bigr|_{y=Y,z=Z},\quad\quad \displaystyle \bar{f}_{z}\approx \frac{\partial f}{\partial z}\Bigr|_{y=Y,z=Z}, \quad\quad \displaystyle \bar{g}_{y}\approx \frac{\partial g}{\partial y}\Bigr|_{y=Y,z=Z}, $$
$$\displaystyle \bar{g}_{z}\approx \frac{\partial g}{\partial z}\Bigr|_{y=Y,z=Z},\quad\quad, \displaystyle \bar{h}_{y}\approx \frac{\partial h}{\partial y}\Bigr|_{y=Y,z=Z},\quad\quad \displaystyle \bar{h}_{z}\approx \frac{\partial h}{\partial z}\Bigr|_{y=Y,z=Z}.$$

Next, we discretize the DAE \cref{eq:adj_all_time} or the ODE \cref{eq:adjoint_ode} using the BDF-1 (Implicit Euler) method.
As we wish to assess the error in the temporal discretization of the \emph{original} DAE, we adopt strategies to reduce the error associated with the \emph{adjoint} problem. 
In particular, we numerically solve the adjoint problem on a more refined temporal grid. 
We define the adjoint time grid points $\tilde{t}_j=t_0+j\Delta \tilde{t}$ for $j=0,1,2,\dots,\tilde{N}$ with step size $\Delta \tilde{t}=\Delta t/r$. 
This means that the adjoint problem is solved with a step size that is $r$ times smaller than step size used to solve the original DAE problem. 
For the experiments presented below, we use $r=4$ except for \cref{ex:hc_model_example} where we use $r=3$ (simply to reduce computational time).
Although the adjoint problem is solved on a finer grid, the adjoint problem (\cref{eq:adj_all_time} or \cref{eq:adjoint_ode}) is linear, and hence does not require a nonlinear solve, which might be required for the numerical solution of the original DAE (\cref{eq:dae_equation} or  \cref{eq:index_2_dae_eq}).
Whenever we require the evaluation of a Jacobian matrix in order to approximate $\bar{f}_{y}, \bar{f}_{z}, \bar{g}_{y}$, $\bar{g}_{z}$, $\bar{h}_{y}$, or $\bar{h}_{z}$ at the adjoint time grid points $\tilde{t}_j$, we use linear interpolation of $Y$ and $Z$ from the original time grid, as discussed in \cref{subsec:ns}.

Both the adjoint problems \cref{eq:adj_all_time,eq:adjoint_ode} are solved backwards in time from $T$ to $0$, using the first order BDF  method discussed in \cref{subsec:ns}. 
Initial conditions (given at time $t=T$, which is the terminal time for the original DAE \cref{eq:dae_equation}) are chosen according to the error representation we wish to assess. 
This produces the approximate solutions $[\overline{\phi}^{(y)},\overline{\phi}^{(z)}]^T$ ($\approx [\phi^{(y)},\phi^{(z)}]^T$) and $[\overline{\nu}^{(y)}, \overline{\nu}^{(z)}]^T$ ($\approx [\nu^{(y)},\nu^{(z)}]^T$), respectively. 


\subsection{Evaluation of the Integrals}
Error estimates in \cref{thm:adj_1_qoi_diff_vari_all_time,thm:adj1_qoi_terminal_time_diff_algeb_vari,thm:error_dae_index2_updated,thm:adj2_qoi_terminal_time_both_diff_alg_vari,,thm:err_using_ode_adj_diff,thm:terminal_error_index1_from_ode_adj} involve integrals of the form 
\begin{equation*}
    \langle a , b \rangle = \int_{0}^{T} \left( a(t), b(t) \right)\,\, dt = \sum_{k=0}^{N-1}\int_{t_k}^{t_{k+1}} \left( a(t), b(t) \right)\,\, dt.
\end{equation*}
We approximate the integral $\displaystyle \int_{t_k}^{t_{k+1}} \left( a(t), b(t) \right)\,\, dt$ by a  5-point Gauss-Legendre quadrature on $[t_{k}, \, t_{k+1}]$.

\subsection{Effectivity Ratio }\label{subsec:effectivity_ratio}

To quantify the accuracy of the error estimates (formed by evaluating the right-hand side of \cref{thm:adj_1_qoi_diff_vari_all_time,thm:adj1_qoi_terminal_time_diff_algeb_vari,thm:adj2_qoi_terminal_time_both_diff_alg_vari,thm:err_using_ode_adj_diff,thm:error_dae_index2_updated,thm:terminal_error_index1_from_ode_adj},   and replacing the adjoint solution by its numerical approximation) we utilize a quantity called the ``effectivity ratio.''
The effectivity ratio is defined as
\begin{align*}
    \text{Effectivity Ratio} = \frac{\text{Error Estimate}}{\text{Reference Error}}. 
\end{align*}

Ideally, the reference error would simply be the difference between the true QoI and the QoI computed using the approximate solution to the DAE (i.e.\ $\eQCI$ or $\eQFT$ depending on context). 
This quantity corresponds to the left hand side of our error representations (\cref{thm:adj_1_qoi_diff_vari_all_time,thm:adj1_qoi_terminal_time_diff_algeb_vari,thm:adj2_qoi_terminal_time_both_diff_alg_vari,thm:err_using_ode_adj_diff,thm:error_dae_index2_updated,thm:terminal_error_index1_from_ode_adj}). 
However, as we do not always have access to the true solution $[y,z]^T$, we typically utilize a numerical solution obtained via standard software packages using very tight tolerances. 
In the following numerical tests, we reduce the original DAE to an ODE via index reduction (\cref{eq:ode_reduced_form}), and then solve this ODE numerically using SciPy's \texttt{solve\_ivp} routine \cite{SciPy-NMeth}. 
We utilize absolute tolerance $\text{atol}=1$e$-15$ and relative tolerance $\text{rtol}=1$e$-12$ to obtain a highly accurate approximation of the true solution, and use this to define the reference error.
However, in \cref{ex:hc_model_example}, we have access to the analytical solution, and therefore avoid this extra layer of approximation.

\subsection{Comparison of the Adjoint DAE and Adjoint ODE Error Estimates}
\label{sec:comparison_approaches}
The Adjoint ODE analysis (\cref{thm:err_using_ode_adj_diff,thm:terminal_error_index1_from_ode_adj}) involves unified results for index-1 and index-2 DAEs, and hence is conceptually simpler than the Adjoint DAE analysis (\cref{thm:adj_1_qoi_diff_vari_all_time,thm:adj1_qoi_terminal_time_diff_algeb_vari,thm:error_dae_index2_updated,thm:adj2_qoi_terminal_time_both_diff_alg_vari}) which has distinct analyses depending on the index of the DAE. However, the error estimates formed using Adjoint ODE analysis incur a higher computational cost than the corresponding ones for  Adjoint DAE analysis. The Adjoint ODE analysis involves evaluating the function $h$ and its Jacobians $\bar{h}_y$ and $\bar{h}_z$, while the corresponding terms in the Adjoint DAE analysis are evaluations of the functions $g$ and its Jacobians $\bar{g}_y$ and $\bar{g}_z$. The function $h$ (see \cref{eq:h_definition}) and its Jacobians may be significantly more expensive to evaluate relative to the function $g$ and its Jacobians. 
Moreover, as \cref{eq:h_definition} makes clear,  evaluating the function $h$ is more expensive for an index-2 DAE than it is for an index-1 DAE. Thus, while the Adjoint ODE analysis may be conceptually simpler, in practice it may be computationally prohibitive, especially for index-2 DAEs with a large dimension (see  \cref{ex:hc_model_example}).

\section{Numerical Results}\label{sec:numerical_results}
We present here a series of numerical tests whereby we evaluate the performance of  error estimates for both semi-explicit index-1 and Hessenberg index-2 DAEs. 
We explore the cases of linear, nonlinear, autonomous, and non autonomous DAEs. 
We also test our error estimate on a system derived from a finite difference discretization of the Electro-Neutral Nernst-Planck Equations (ENNPE), which form a system of nonlinear 
Partial Differential Algebraic Equations (PDAEs). 
In all cases, we solve the DAE in question (\cref{eq:dae_equation} or  \cref{eq:index_2_dae_eq}) using the first order BDF method as described in \cref{subsec:ns}. In our implementation, we use the Python routine scipy.optimize (fsolve) \cite{SciPy-NMeth} to solve the implicit systems given by \cref{eq:bdf1} and \cref{eq:bdf1_for_index2}. 
Initial conditions are chosen to ensure consistency (i.e. to satisfy the explicit constraint equation and the hidden constraint, if any).

In \cref{sec:index_one_example_sec,sec:hessenberg_2_example_sec} we present results for index-1 and Hessenberg index-2 DAEs. A DAE arising from a PDE with constraints is investigated in \cref{sec:ENNPE}. The tables in this section use the notation Adjoint DAE to refer to the estimates developed from \cref{sec:adj_error_analysis_sec_adj_dae} (\cref{thm:adj_1_qoi_diff_vari_all_time,thm:adj1_qoi_terminal_time_diff_algeb_vari,thm:error_dae_index2_updated,thm:adj2_qoi_terminal_time_both_diff_alg_vari}), and use Adjoint ODE to refer to the estimates developed from \cref{sec:adj_error_analysis_sec_adj_ode} (\cref{thm:err_using_ode_adj_diff,thm:terminal_error_index1_from_ode_adj}).

\subsection{Semi-explicit Index-1 DAEs Examples}\label{sec:index_one_example_sec}
We begin with  relatively simple, index-1 DAEs in semi-explicit form. 

\begin{example}\label{ex:robertson_index_1_DAE_example}
The Robertson semi-explicit index-1 DAEs is 
\begin{equation}\label{eq:robertson_equation_index_1}
    \left. \begin{aligned}
        \dot{y}_{1} &=-0.04y_{1}+10^4y_{2}z,\\
      \dot{y}_{2} &=0.04y_{1}-10^4y_{2}z-(3\times 10^7)y_{2}^2,\\
      0 &=y_{1}+y_{2}+z-1,
    \end{aligned}\right\}
\end{equation}
with the initial condition for the differential variable $y(0) = [1, 0]^T$. The  consistent initial condition for the algebraic variable is $z(0)=0$ and $t\in (0, T]$. This is a common example problem discussed in the documentation of existing numerical DAE solvers \cite{robertsonexampletaken}. 
\end{example}
For this example, we focus on a QoI that is cumulative over time, $\mathcal{Q}_{\left[0,T\right]}$.
Below, we present the error estimate and effectivity ratios for two variations of $\QCI{y}{z}$, one involving only the differential variables (\cref{tab:numerical_results_for_robertson_ode_adj_and_dae_adj_cumulative_diff}) and one involving only the algebraic variables (\cref{tab:numerical_results_for_robertson_ode_adj_and_dae_adj_cumulative_alg}). 
Both are cumulative over the interval $[0,T]$, and computed using \cref{thm:adj_1_qoi_diff_vari_all_time} and \cref{thm:err_using_ode_adj_diff}.
We first set $\psi^y(t)=[1,1]^T$, $\psi^z(t)=0$ and assess our error estimate using both the Adjoint DAE and Adjoint ODE approaches. 
The results are presented in \cref{tab:numerical_results_for_robertson_ode_adj_and_dae_adj_cumulative_diff}. We then repeat the experiment with $\psi^z(t)=1$, $\psi^y(t)=[0,0]^T$ and present the results in \cref{tab:numerical_results_for_robertson_ode_adj_and_dae_adj_cumulative_alg}. 
Both experiments were performed for multiple values of terminal time ($T$) and time step size of the original DAE ($\Delta t$). 
For clarity of presentation, we only report our results to $5$ significant digits.

The data in \cref{tab:numerical_results_for_robertson_ode_adj_and_dae_adj_cumulative_diff,tab:numerical_results_for_robertson_ode_adj_and_dae_adj_cumulative_alg} 
indicates that the error estimates from both approaches are quite accurate, and effectivity ratio close to one, which illustrates the validation of our theorems.  Moreover, the tables 
appear to indicate that there is no difference between the errors calculated using the adjoint DAE and adjoint ODE approaches. 
 This is not true, strictly speaking, as the results corresponding to adjoint ODEs and Adjoint DAEs display a difference after the tenth digit (not shown). 

\begin{table}[ht]
    \centering
    \caption{Numerical Results for  DAE in \cref{ex:robertson_index_1_DAE_example} using Adjoint DAE (\cref{thm:adj_1_qoi_diff_vari_all_time}), and Adjoint ODE (\cref{thm:err_using_ode_adj_diff}) 
    to estimate the error    $\eQCI$ where   $\psi^{y}(t) = [1,1]^T$ and $\psi^{z}(t) = 0$.    
    }
    \label{tab:numerical_results_for_robertson_ode_adj_and_dae_adj_cumulative_diff}
    \begin{tabular}{cccccc}
        \toprule
        \( \Delta t \) & \( T \) & \multicolumn{2}{c}{Error Estimate} & \multicolumn{2}{c}{Effectivity Ratio}\\
        \cmidrule(lr){3-4}
        \cmidrule(lr){5-6}
        & & Adjoint ODEs & Adjoint DAEs & Adjoint ODEs & Adjoint DAEs \\
        \midrule
        0.001 & 1 & -2.8546e-06 & -2.8546e-06 & 0.9989 & 0.9989 \\
              & 10 & -6.4758e-05 & -6.4758e-05 & 0.9999 & 0.9999 \\
              & 20 & -1.2082e-04 & -1.2082e-04 & 0.9999 & 0.9999 \\
              & 50 & -2.3590e-04 & -2.3590e-04 & 0.9999 & 0.9999 \\
              & 100 & -3.5917e-04 & -3.5917e-04 & 0.9999 & 0.9999 \\
        \midrule      
        0.0005 & 1 & -1.4288e-06 & -1.4288e-06 & 0.9996 & 0.9996 \\
               & 10 & -3.2384e-05 & -3.2384e-05 & 0.9999 & 0.9999 \\
               & 20 & -6.0415e-05 & -6.0415e-05 & 1.0 & 1.0 \\
               & 50 & -1.1796e-04 & -1.1796e-04 & 1.0 & 1.0 \\
               & 100 & -1.7960e-04 & -1.7960e-04 & 1.0 & 1.0 \\   
        \bottomrule
    \end{tabular}
\end{table}

\begin{table}[ht]
    \centering
    \caption{Numerical Results for  DAE in \cref{ex:robertson_index_1_DAE_example} using Adjoint DAE ( \cref{thm:adj_1_qoi_diff_vari_all_time}), and Adjoint ODE (\cref{thm:err_using_ode_adj_diff}) to estimate the error $\eQCI$ where $\psi^{z}(t) = 1$ and $\psi^{y}(t) = [0,0]^T$.}
    \label{tab:numerical_results_for_robertson_ode_adj_and_dae_adj_cumulative_alg}
    \begin{tabular}{cccccc}
        \toprule
        \( \Delta t \) & \( T \) & \multicolumn{2}{c}{Error Estimate} & \multicolumn{2}{c}{Effectivity Ratio}\\
        \cmidrule(lr){3-4}
        \cmidrule(lr){5-6}
        & & Adjoint ODEs & Adjoint DAEs & Adjoint ODEs & Adjoint DAEs \\
        \midrule
        0.001 & 1 & 2.8546e-06 & 2.8546e-06 & 0.9989 & 0.9989 \\
              & 10 & 6.4758e-05 & 6.4758e-05 & 0.9999 & 0.9999 \\
              & 20 & 1.2082e-04 & 1.2082e-04 & 0.9999 & 0.9999 \\
              & 50 & 2.3590e-04 & 2.3590e-04 & 0.9999 & 0.9999 \\
              & 100 & 3.5917e-04 & 3.5917e-04 & 0.9999 & 0.9999 \\
        \midrule      
        0.0005 & 1 & 1.4288e-06 & 1.4288e-06 & 0.9996 & 0.9996 \\
               & 10 & 3.2384e-05 & 3.2384e-05 & 0.9999 & 0.9999 \\
               & 20 & 6.0415e-05 & 6.0415e-05 & 1.0 & 1.0 \\
               & 50 & 1.1796e-04 & 1.1796e-04 & 1.0 & 1.0 \\
               & 100 & 1.7960e-04 & 1.7960e-04 & 1.0 & 1.0 \\    
        \bottomrule
    \end{tabular}
\end{table}

For our next example, we again analyze an index-1 DAE in semi-explicit form, but with a nonlinear algebraic constraint. 
\begin{example}\label{ex:pendulum_index_1_DAE_example}
The Pendulum semi-explicit index-1 DAEs \cite{pendulumexampletaken} is
\begin{equation}\label{eq:pendulum_index_1_equation}
    \left. \begin{aligned}
        \dot{y}_1 &= y_3,\\
      \dot{y}_2 &= y_4,\\
      m\dot{y}_3 &= -2y_1z,\\
      m\dot{y}_4 &= - mg - 2y_2z,\\
      m\left(y_3^2 + y_4^2 - gy_2\right) - 2z\left(y_1^2 + y_2^2\right) &= 0,
    \end{aligned} \right \}
\end{equation}
for $t\in (0,T]$. Here the parameters  are $m=1,g=9.81, s=1$, and the initial condition for the differential variable is $y(0)=[0,-s,1,0]^T$. The  consistent initial condition for the algebraic variable is  $z (0) = \frac{m(1+s\cdot g)}{2s^2}$ and $t \in (0,T]$.
\end{example}
For this example we discuss results for two variations of QoI at the terminal time, $\mathcal{Q}_T$ (see \cref{eq:qoi_def_at_terminal_time});  one involving only the differential variables (\cref{tab:numerical_results_for_pendulum1_ode_adj_and_dae_adj_terminal_diff}) and one involving only the algebraic variables (\cref{tab:numerical_results_for_pendulum1_ode_adj_and_dae_adj_terminal_alg}). 
Error estimates for  various values of the terminal time $(T)$ and time step size ($\Delta t$) are computed using \cref{thm:adj1_qoi_terminal_time_diff_algeb_vari} and \cref{thm:terminal_error_index1_from_ode_adj}.
In \cref{tab:numerical_results_for_pendulum1_ode_adj_and_dae_adj_terminal_diff} we show error estimates and effectivity ratios for the case when the QoI depends on the differential variables only, that is, $\QFT{y}{z} = \left( y(T), \zeta^y \right)$, with $\zeta^y=[1,1,1,1]^T$ and $\zeta^z=0$.
We show results for the case when the QoI depends on the algebraic variable only, $\QFT{y}{z} = \left( z(T), \zeta^z \right)$, with $\zeta^z=1$ and $\zeta^y=[0,0,0,0]^T$ in \cref{tab:numerical_results_for_pendulum1_ode_adj_and_dae_adj_terminal_alg}.
Again, effectivity ratios demonstrates the accuracy of the error estimates.

\begin{table}[ht]
    \centering
    \caption{Numerical Results for  DAE in \cref{ex:pendulum_index_1_DAE_example} using Adjoint DAE (\cref{thm:adj1_qoi_terminal_time_diff_algeb_vari}), and Adjoint ODE (\cref{thm:terminal_error_index1_from_ode_adj}) to estimate the error $\eQFT$ where $\zeta^{y}= [1,1,1,1]^T $ and $\zeta^{z} = 0$.}
    \label{tab:numerical_results_for_pendulum1_ode_adj_and_dae_adj_terminal_diff}
    \begin{tabular}{cccccc}
        \toprule
        \( \Delta t \) & \( T \) & \multicolumn{2}{c}{Error Estimate} & \multicolumn{2}{c}{Effectivity Ratio}\\
        \cmidrule(lr){3-4}
        \cmidrule(lr){5-6}
        & & Adjoint ODEs & Adjoint DAEs & Adjoint ODEs & Adjoint DAEs \\
        \midrule
        0.001 & 1 & -5.0252e-03 & -5.0234e-03 & 0.9997 & 0.9993 \\
              & 2 & 9.1300e-03 & 9.1376e-03 & 0.9986 & 0.9994 \\
              & 3 & -1.5065e-02 & -1.5059e-02 & 0.9975 & 0.9972 \\
              & 4 & 1.6388e-02 & 1.6414e-02 & 1.0 & 1.002 \\
              & 5 & -2.4788e-02 & -2.4786e-02 & 0.9952 & 0.9951 \\
        \midrule      
        0.0005 & 1 & -2.5171e-03 & -2.5166e-03  & 0.9999 & 0.9997 \\
               & 2 & 4.5715e-03 & 4.5734e-03 & 0.9993 & 0.9997 \\
               & 3 & -7.5751e-03  & -7.5738e-03 & 0.9988 & 0.9986 \\
               & 4 & 8.2010e-03 & 8.2074e-03 & 1.0 & 1.001 \\
               & 5 & -1.2512e-02 & -1.2511e-02 & 0.9976 & 0.9976 \\ 
        \bottomrule
    \end{tabular}
\end{table}

\begin{table}[ht]
    \centering
    \caption{Numerical Results for  DAE in \cref{ex:pendulum_index_1_DAE_example} using Adjoint DAE (\cref{thm:adj1_qoi_terminal_time_diff_algeb_vari}), and Adjoint ODE (\cref{thm:terminal_error_index1_from_ode_adj}) to estimate the error $\eQFT$ where $\zeta^z = 1$ and $\zeta^y = [0,0,0,0]^T$.}
    \label{tab:numerical_results_for_pendulum1_ode_adj_and_dae_adj_terminal_alg}
    \begin{tabular}{cccccc}
        \toprule
        \( \Delta t \) & \( T \) & \multicolumn{2}{c}{Error Estimate} & \multicolumn{2}{c}{Effectivity Ratio}\\
        \cmidrule(lr){3-4}
        \cmidrule(lr){5-6}
        & & Adjoint ODEs & Adjoint DAEs & Adjoint ODEs & Adjoint DAEs \\
        \midrule
        0.001 & 1 & 5.0019e-03 & 5.0059e-03 & 0.9969 & 0.9977 \\
              & 2 & 9.8723e-03 & 9.8878e-03 & 0.9939 & 0.9954 \\
              & 3 & 1.4554e-02 & 1.4587e-02 & 0.9912 & 0.9934 \\
              & 4 & 1.8998e-02 & 1.9050e-02 & 0.9888 & 0.9915 \\
              & 5 & 2.3160e-02 & 2.3230e-02 & 0.9871 & 0.9901 \\
        \midrule      
        0.0005 & 1 & 2.5094e-03 & 2.5104e-03 & 0.9984 & 0.9988 \\
               & 2 & 4.9689e-03 & 4.9728e-03 & 0.9969 & 0.9977 \\
               & 3 & 7.3475e-03 & 7.3556e-03 & 0.9956 & 0.9967 \\
               & 4 & 9.6163e-03 & 9.6293e-03 & 0.9944 & 0.9958 \\
               & 5 & 1.1749e-02 & 1.1766e-02 & 0.9936 & 0.9951 \\
        \bottomrule
    \end{tabular}
\end{table}

\subsection{Hessenberg (Pure) Index-2 DAEs 
Examples}\label{sec:hessenberg_2_example_sec}
We now move on to a nonlinear non-autonomous index-2 DAE in Hessenberg form.

\begin{example}\label{ex:index2_picked_from_Petzol_computer_method}
The following system is a Hessenberg (pure) index-2 DAE
\begin{equation}\label{eq:index2 picked from Petzol's computer method}
    \left. \begin{aligned}
    \dot{y}_1 &= \lambda y_1 - z, \\
    \dot{y}_2 &= (2\lambda - \sin^2 t) y_2 + (\sin^2 t) (y_1 - 1)^2, \\
    0 &= y_2 - (y_1 - 1)^2,
\end{aligned}\right \}    
\end{equation}
where $\lambda$ is a parameter and the given initial condition corresponding to the differential variable is $y(0) =[2, 1]^T$. The  consistent initial condition for the algebraic variable is $z(0) = \lambda$ and $t\in (0,T]$.
\end{example}

For this example, we place our attention on a QoI that is cumulative time, $\mathcal{Q}_{\left[0,T\right]}$ (see \cref{eq:qoi_def_inverval}). For this experiment we set $\lambda=1$.
Once again, we present the error estimate and effectivity ratios for two variations of $\mathcal{Q}_{\left[0,T\right]}$, one involving only the differential variables (\cref{tab:numerical_results_for_index2_from_petzol_book_ode_adj_and_dae_adj_terminal_diff}) and one involving only the algebraic variables (\cref{tab:numerical_results_for_index2_from_petzol_book_ode_adj_and_dae_adj_terminal_alg}). 
The error estimates are computed using \cref{thm:error_dae_index2_updated} and \cref{thm:err_using_ode_adj_diff}.
We first set $\psi^y(t)=[1,1]^T$, $\psi^z(t)=0$ and asses our error estimate using both the Adjoint DAE and Adjoint ODE approaches in \cref{tab:numerical_results_for_index2_from_petzol_book_ode_adj_and_dae_adj_terminal_diff}. We then repeat the experiment with $\psi^z(t)=1$, $\psi^y(t)=[0,0]^T$ and present the results in \cref{tab:numerical_results_for_index2_from_petzol_book_ode_adj_and_dae_adj_terminal_alg}. 
Once again we see small error estimates and effectivity ratios very close to unity. 
The only exception is the error estimate using the adjoint ODE, where as the terminal time $T$ increases, we see that the error estimate becomes slightly less accurate as indicated by the effectivity ratio.

\begin{table}[ht]
    \centering
    \caption{Numerical Results for  DAE in \cref{ex:index2_picked_from_Petzol_computer_method} using Adjoint DAE (\cref{thm:error_dae_index2_updated}), and Adjoint ODE (\cref{thm:err_using_ode_adj_diff}) to estimate the error    $\eQCI$ where   $\psi^{y}(t) = [1,1]^T$ and $\psi^{z}(t) = 0$.}
    \label{tab:numerical_results_for_index2_from_petzol_book_ode_adj_and_dae_adj_terminal_diff}
    \begin{tabular}{cccccc}
        \toprule
        \( \Delta t \) & \( T \) & \multicolumn{2}{c}{Error Estimate} & \multicolumn{2}{c}{Effectivity Ratio}\\
        \cmidrule(lr){3-4}
        \cmidrule(lr){5-6}
        & & Adjoint ODEs & Adjoint DAEs & Adjoint ODEs & Adjoint DAEs \\
        \midrule
        0.001 & 1 & -5.6088e-04 & -5.6076e-04 & 0.9999 & 0.9997 \\
              & 2 & -1.0454e-03 & -1.0473e-03 & 0.9978 & 0.9996 \\
              & 3 & -1.2735e-03 & -1.2912e-03 & 0.9859 & 0.9996 \\
        \midrule      
        0.0005 & 1 & -2.8053e-04 & -2.8050e-04 & 1.0 & 0.9999 \\
               & 2 & -5.2340e-04 & -5.2389e-04 & 0.9989 & 0.9998 \\
               & 3 & -6.4142e-04 & -6.4583e-04 & 0.9989 & 0.9998 \\
        \bottomrule
    \end{tabular}
\end{table}

\begin{table}[ht]
    \centering
    \caption{Numerical Results for  DAE in \cref{ex:index2_picked_from_Petzol_computer_method} using Adjoint DAE (\cref{thm:error_dae_index2_updated}), and Adjoint ODE (\cref{thm:err_using_ode_adj_diff}) to estimate the error    $\eQCI$ where   $\psi^{z}(t) = 1$ and $\psi^{y}(t) = [0,0]^T$.}
    \label{tab:numerical_results_for_index2_from_petzol_book_ode_adj_and_dae_adj_terminal_alg}
    \begin{tabular}{cccccc}
        \toprule
        \( \Delta t \) & \( T \) & \multicolumn{2}{c}{Error Estimate} & \multicolumn{2}{c}{Effectivity Ratio}\\
        \cmidrule(lr){3-4}
        \cmidrule(lr){5-6}
        & & Adjoint ODEs & Adjoint DAEs & Adjoint ODEs & Adjoint DAEs \\
        \midrule
        0.001 & 1 & 3.1533e-04 & 3.1540e-04 & 0.9991 & 0.9993 \\
              & 2 & 4.2383e-04 & 4.3143e-04 & 0.9812 & 0.9988 \\
              & 3 & 4.2758e-04 & 4.7408e-04 & 0.9006 & 0.9985 \\
        \midrule      
        0.0005 & 1 & 1.5785e-04 & 1.5786e-04 & 0.9995 & 0.9996 \\
               & 2 & 2.1404e-04 & 2.1594e-04 & 0.9906 & 0.9994 \\
               & 3 & 2.2564e-04 & 2.3730e-04 & 0.9502 & 0.9993 \\
        \bottomrule
    \end{tabular}
\end{table}

\begin{example}\label{ex:pendulum_index_2_DAE_example}
The Pendulum index-2 DAE \cite{pendulumexampletaken} system is
\begin{equation}\label{eq:pendulum_index_2_equation}
   \left. \begin{aligned}
          \dot{y}_1 &= y_3,\\
          \dot{y}_2 &= y_4,\\
          m\dot{y}_3 &= -2y_1z,\\
          m\dot{y}_4 &= - mg - 2y_2z,\\
          y_1y_3 + y_2y_4 &= 0,
    \end{aligned}\right\}
\end{equation}
for $t\in (0,T]$. Here the parameters  are $m=1,g=9.81, s=1$, and the initial condition for the differential variable is $y(0)=[0,-s,1,0]^T$. The  consistent initial condition for the algebraic variable is  $z (0) = \frac{m(1+s\cdot g)}{2s^2}$ and $t \in (0,T]$.
\end{example}
The pendulum index-2 DAEs in \cref{ex:pendulum_index_2_DAE_example} represent a Hessenberg index-2 problem. Here, we concentrate on the QoI at the terminal time, $\mathcal{Q}_T$ (see \cref{eq:qoi_def_at_terminal_time}). We present in \cref{tab:numerical_results_for_pendulum_2_ode_adj_and_dae_adj_terminal_diff} the error estimate and effectivity ratios for $\mathcal{Q}_T$, involving both the differential variables and algebraic variables. The error estimates are computed using \cref{thm:adj2_qoi_terminal_time_both_diff_alg_vari} and \cref{thm:terminal_error_index1_from_ode_adj}. The QoI $\mathcal{Q}_T$ is defined by  setting $\zeta^y=[1,1,1,1]^T, \zeta^z=1$.
We observe promising effectivity ratios for both the Adjoint ODE and Adjoint DAE approaches.

\begin{table}[ht]
    \centering
    \caption{Numerical Results for  DAE in \cref{ex:pendulum_index_2_DAE_example} using Adjoint DAE (\cref{thm:adj2_qoi_terminal_time_both_diff_alg_vari}), and Adjoint ODE (\cref{thm:terminal_error_index1_from_ode_adj}) to estimate the error $\eQFT$ where $\zeta^y = [1,1,1,1]^T$ and $\zeta^z = 1$.}
    \label{tab:numerical_results_for_pendulum_2_ode_adj_and_dae_adj_terminal_diff}
    \begin{tabular}{cccccc}
        \toprule
        \( \Delta t \) & \( T \) & \multicolumn{2}{c}{Error Estimate} & \multicolumn{2}{c}{Effectivity Ratio}\\
        \cmidrule(lr){3-4}
        \cmidrule(lr){5-6}
        & & Adjoint ODEs & Adjoint DAEs & Adjoint ODEs & Adjoint DAEs \\
        \midrule
        0.001 & 1 & -1.7159e-03 & -1.7153e-03 & 1.003 & 1.002 \\
              & 2 & 1.5162e-02 & 1.5169e-02 & 0.9975 & 0.998 \\
              & 3 & -5.5532e-03 & -5.5430e-03 & 1.006 & 1.004 \\
              & 4 & 2.6268e-02 & 2.6288e-02 & 0.999 & 0.9997 \\
              & 5 & -9.9657e-03 & -9.9507e-03 & 1.002 & 1.0 \\
              
        \midrule      
        0.0005 & 1 & -8.5615e-04 & -8.5598e-04 & 1.001 & 1.001 \\
               & 2 & 7.6081e-03 & 7.6099e-03 & 0.9988 & 0.999 \\
               & 3 & -2.7685e-03 & -2.7659e-03 & 1.003 & 1.002  \\
               & 4 & 1.3182e-02 & 1.3187e-02 & 0.9996 & 0.9999 \\
               & 5 & -4.9932e-03 & -4.9893e-03 & 1.001 & 0.9999 \\
              
        \bottomrule
    \end{tabular}
\end{table}

\subsection{Electro-Neutral Nernst-Planck Equation  of Ion Transport}
\label{sec:ENNPE}
\begin{example}\label{ex:hc_model_example}
We consider a system describing the evolution of the concentrations of two monovalent ions. This system is sometimes called the Electro-Neutral Nernst-Planck Equations (ENNPE) \cite{profowensmodel, NernstPlanck}. The evolution equations and constraint are written as
    \begin{equation}\label{eq:hc_model_equation}
    \left.\begin{aligned}
    \frac{\partial c}{\partial t} &= D_{c}\left(\frac{\partial^2 c}{\partial x^2} +\frac{\partial }{\partial x}\left(c\frac{\partial \Psi}{\partial x}\right)\right),\\
    \frac{\partial a}{\partial t} &= D_{a}\left(\frac{\partial^2 a}{\partial x^2} -\frac{\partial }{\partial x}\left(a\frac{\partial \Psi}{\partial x} \right)\right),\\
    0 &= c-a ,
\end{aligned}\right \}
\end{equation}
where $c(x,t)$ and $a(x,t)$ denote the concentration of Hydrogen ion (cation) and Chloride ion (anion) respectively, and $\frac{\partial \Psi}{\partial x}(x,t)$ is the electric potential gradient.
No flux (homogeneous Robin) boundary conditions are given by
\begin{equation}\label{eq:robin_bcs}
    \left.\begin{aligned}
        \frac{\partial c}{\partial x} + c\frac{\partial \Psi}{\partial x}\Bigr|_{x=0} &= \frac{\partial c}{\partial x} + c\frac{\partial \Psi}{\partial x}\Bigr|_{x=1} &= 0,\\
        \frac{\partial a}{\partial x} - a\frac{\partial \Psi}{\partial x}\Bigr|_{x=0} &= \frac{\partial a}{\partial x} - a\frac{\partial \Psi}{\partial x}\Bigr|_{x=1} &= 0.\\
    \end{aligned}\right \}
\end{equation}
   Initial conditions are $c(x,0) = a(x,0) = 2 + cos(\pi x)$ and $x\in [0,1], \, t\in (0,T]$, where $D_{c}, D_{a}$ are the diffusion coefficients of the cation and anion ions respectively.
\end{example}
The method of separation of variables gives the analytic solution
\begin{align}\label{eq:true_solution_hc_model}
    c(x,t) &=a(x,t)=2+e^{-\pi^2 D_
    \text{eff}t}\cos (\pi x),\\
    w(x,t) &= \left(\frac{D_{a}-D_{c}}{D_{a} + D_{c}}\right)\frac{-\pi e^{-\pi^2 D_\text{eff}t}\sin (\pi x)}{2+e^{-\pi^2 D_\text{eff}t}\cos (\pi x)}\label{eq:true solution W},
\end{align}
where
$$D_\text{eff}=\frac{2D_{c}D_{a}}{D_{c}+D_{a}},\text{ and } w =\frac{\partial \Psi}{\partial x}.$$
Notice that the electric potential $\Psi$ is only defined up to an additive constant. 
Therefore we utilize $w = \frac{\partial \Psi}{\partial x}$ as our state variable. 

\subsubsection{Spatial discretization: Staggered Grid Approach}\label{subsec:semi-discretization} 
We use a staggered grid to discretize the spatial derivatives. 
This technique has been used in many numerical methods (particularly in computational fluid dynamics) since the development of the Marker and Cell (MAC) method \cite{mckee2008:kjd}.
Consider a set of $N_s-1$ uniformly spaced grid points in the spatial domain with spacing $\Delta x$. 
For the unit interval let,
    $$\tilde{x}_i= i\Delta x,\quad\quad i=1,2,\dots ,N_s-1,\quad\quad\Delta x = \frac{1}{N_s}. $$
These points are often called ``cell edges'' and are the locations where we represent vector valued quantities such as the electric potential gradient $w(\tilde{x}_i,t)$. 
We also introduce a second spatial grid at the so-called ``cell centers''
$$x_j = (2j-1)\Delta x/2,\quad j = 1,2,\dots,N_s.$$
These points will be used to represent the values of scalar valued quantities such as the ion concentrations ($c(x_j,t)$ and $a(x_j,t)$). 
This staggered grid facilitates the use of centered finite difference approximations to the various spatial derivatives in \cref{eq:hc_model_equation} and representation of the no flux boundary conditions \cref{eq:robin_bcs}, and leads to the approximation:  $C_j(t) \approx c(x_j,t)$, $A_j(t) \approx a(x_j,t)$, and $W_i(t) \approx w(\tilde{x}_i,t)$. 
Details of the staggered grid approach  are given in appendix \ref{macscheme}. 
Spatial disretization converts the ENNPE to a Hessenberg index-2 DAE,

\begin{equation}\label{eq:hc_index_2_equation}
   \left. \begin{aligned}
          \dot{C} &= D_{c} M C + D_{c} B_C,\\
    \dot{A} &= D_{a} M A - D_{c} B_A,\\
    0 &= \Pi(C) - \Pi(A)  ,
    \end{aligned}\right\}
\end{equation}
where, 
\[
M = \frac{1}{\Delta x^2} \begin{pmatrix}
-1 & 1 & 0 & \cdots & 0 \\
1 & -2 & 1 & \cdots & 0 \\
0 & 1 & -2 & \cdots & 0 \\
\vdots & \vdots & \vdots & \ddots & \vdots \\
0 & 0 & 0 & \cdots & -1
\end{pmatrix},
B_C = \frac{1}{2\Delta x}\begin{pmatrix}
(C_1 + C_2)W_1 \\
(C_2 + C_3)W_2 - (C_1 + C_2)W_1 \\
(C_3 + C_4)W_3 - (C_2 + C_3)W_2 \\
\vdots \\
- (C_{N_s-1} + C_{N_s})W_{N_s-1}
\end{pmatrix},
\]
\[
B_A = \frac{1}{2\Delta x}\begin{pmatrix}
(A_1 + A_2)W_1 \\
(A_2 + A_3)W_2 - (A_1 + A_2)W_1 \\
(A_3 + A_4)W_3 - (A_2 + A_3)W_2 \\
\vdots \\
- (A_{N_s-1} + C_{N_s})W_{N_s-1}
\end{pmatrix},
C = \begin{pmatrix}
    C_1\\
    C_2\\
    \vdots\\
    C_{N_s}
\end{pmatrix},
A = \begin{pmatrix}
    A_1\\
    A_2\\
    \vdots\\
    A_{N_s}
\end{pmatrix},
\]
and
\[
W = \begin{pmatrix}
    W_1\\
    W_2\\
    \vdots\\
    W_{N_s-1}
\end{pmatrix},
\Pi(\xi) = \begin{pmatrix}
    \xi_2\\
    \xi_3\\
    \vdots\\
    \xi_{N_s}
\end{pmatrix} \text{ for }
\xi = \begin{pmatrix}
    \xi_1\\
    \xi_2\\
    \vdots\\
    \xi_{N_s}
\end{pmatrix},
\]
where $M\in \mathbb{R}^{N_s\times N_s}$, and $B_C, B_A, C(t), A(t), \xi(t)\in \mathbb{R}^{N_s}$, and $W(t)\in \mathbb{R}^{N_s-1}$.
We set the differential variables as $y = [C^T,A^T]^T$, and algebraic variables as $z=W$, with $n=2N_s$, and $m=N_s-1$, then \cref{eq:hc_index_2_equation} is exactly in the form of \cref{eq:index_2_dae_eq}.
Initial conditions corresponding to the differential variables are given by
$$C_j(0)=A_j(0)=2+\cos{(\pi x_j)},\quad \text{ for } j=1,2,3,\dots, N_s.$$

There is a subtle point to be made with regards to the initial condition for the electric potential gradient (the algebraic variable). 
For the PDE system \cref{eq:hc_model_equation}, one can obtain the initial condition by differentiating the constraint equation, substituting the transport equations to eliminate $\frac{\partial c}{\partial t}$ and $\frac{\partial a}{\partial t}$, solving for $w(x,t)$ (this calculation is used to derive \cref{eq:true solution W}) and evaluating at $t=0$ to yield 
\begin{equation} w(x,0) = \left. \frac{D_c \frac{\partial c}{\partial x}- D_a \frac{\partial a}{\partial x}}{D_c c + D_a a}\right|_{t=0}.\label{eq:analytic_IC}
\end{equation}
This expression can then be evaluated at the points $\tilde{x}_i$ to derive what we call the ``analytic initial condition'' for the algebraic variables $W_i$. 

Alternately, one can perform a similar calculation on the spatially discretized DAE system \cref{eq:hc_index_2_equation}: differentiate the constraint equation, eliminate $\dot C$ and $\dot A$, solve for $W$ and set $t=0$. 
Doing so leads to what we call the ``discrete initial condition'' corresponding to the algebraic variables $W_i$,
\begin{align}\label{eq:semi_discrete_IC}
    W_l(0) = \frac{D_{c}\left(\frac{C_l-C_{l+1}}{\Delta x}\right) - D_{a}\left(\frac{A_l-A_{l+1}}{\Delta x}\right)}{D_{c}\left(\frac{C_l + C_{l+1}}{2}\right) + D_{a}\left(\frac{A_l + A_{l+1}}{2}\right)}\Bigr|_{t=0}, \quad \text{for } l=1,2,3,\cdots, N_s-1.
\end{align}
Notice that \cref{eq:semi_discrete_IC} is equivalent to a relatively simple centered finite difference approximation to \cref{eq:analytic_IC}. 
They are not equivalent, but differ by the spatial discretization error. 
This does raise interesting questions about the appropriate initial conditions that one should utilize when time-evolving discretized PDAEs, and indicates a potential line of future inquiry. 
In our numerical experiments, we used both the discrete initial condition and analytic initial condition for algebraic variables, and 
observed no appreciable difference in the error analysis that is the focus of this paper. 
For brevity we present only the results obtained using the discrete initial condition. 

\subsubsection{Error analysis of the ENNPE model}\label{subsec:error-analysis-ennpe}
For this semi-discretized PDAE (ENNPE) \cref{eq:hc_index_2_equation}, we focus on the QoI at the terminal time, $\mathcal{Q}_T$ (see \cref{eq:qoi_def_at_terminal_time}). We present the error estimate based on Adjoint DAE and effectivity ratios for two variations of $\mathcal{Q}_T$, one involving only the differential variables (column corresponding to $\left ( e^{(y)}(T),\zeta^y\right)$ in \cref{tab:numerical_results_for_HC_model_catastrophic_error}) and one involving only the algebraic variables (column corresponding to $\left ( e^{(z)}(T),\zeta^z\right)$ in \cref{tab:numerical_results_for_HC_model_catastrophic_error}).

The error in the QoI $\mathcal{Q}_T$ involving both the differential and algebraic state variables at $t=T$ is estimated using Adjoint DAE (\cref{thm:adj2_qoi_terminal_time_both_diff_alg_vari}). We choose the spatial grid spacing $\Delta x=0.004$, which  is small enough so that the error due to spatial discretization is negligible relative to the error due to temporal discretization. The dimension of the resulting nonlinear Hessenberg index-2 DAE is $749$. We are not performing Adjoint ODE based error estimation for this particular PDAE because of its too high computation cost to obtain the corresponding index reduction ODE system.
%
The semi-discretized system in \cref{eq:hc_index_2_equation} corresponds to differential variables $y = [C^T,A^T]^T$, and algebraic variables $z=W$. The values used for the diffusion coefficients are   $D_c=0.5$, and $D_a=0.05$.

 To estimate the error $\left ( e^{(y)}(T),\zeta^y\right)$ (using \cref{thm:adj2_qoi_terminal_time_both_diff_alg_vari}) involving only the differential variables $y$ at $t=T$, we set $\zeta^y=[\mathbf{1}^{N_s},\mathbf{1}^{N_s}]^T$, where $\mathbf{1}^{N_s} = [1,1,\ldots,1] \in \mathbb{R}^{N_s}$ is a column vector of ones (we also set $\zeta^z=[\mathbf{0}^{N_s-1}]^T$, where $\mathbf{0}^{N_s-1} = [0,0,\ldots,0] \in \mathbb{R}^{N_s-1}$ is a column vector of zeros). 
 Similarly, to evaluate  $\left ( e^{(z)}(T),\zeta^z\right)$  using \cref{thm:adj2_qoi_terminal_time_both_diff_alg_vari}), we set $\zeta^z=[\mathbf{1}^{N_s-1}]^T$, and $\zeta^y=[\mathbf{0}^{N_s},\mathbf{0}^{N_s}]^T$.
The reference error is computed using the analytical solution \cref{eq:true_solution_hc_model,eq:true solution W} evaluated at terminal time and the results are presented in \cref{tab:numerical_results_for_HC_model_catastrophic_error}

We notice that the estimates of error $\left ( e^{(z)}(T),\zeta^z\right)$ are quite accurate whereas those for the error $\left ( e^{(y)}(T),\zeta^y\right)$ demonstrate somewhat anomalous results. The effectivity ratios for this latter case are sometimes not close to one. 
In precisely these cases, the errors are extremely small (within a couple of orders of magnitude of machine precision), and hence incur catastrophic floating point cancellation in their calculation. 
To illustrate this, we split the integral $\langle \phi^{(y)},f(Y,Z)-\dot{Y} \rangle$ into four integrals  as follows:
\begin{align*}
    \langle \phi^{(y)},f(Y,Z)-\dot{Y}\rangle = I_1+I_2+I_3+I_4,
\end{align*}
where $I_i=\langle \phi^{(y)}_i,\tilde{f}_i \rangle$, each of $\phi^{(y)}_i$'s, $\tilde{f}_i$'s are vectors of same size (125) for $i=1,2,3,4$ and analogously $\phi^{(y)}=[\phi^{(y)}_1,\phi^{(y)}_2,\phi^{(y)}_3,\phi^{(y)}_4]^T$, and $f(Y,Z)-\dot{Y}=[\tilde{f}_1,\tilde{f}_2,\tilde{f}_3,\tilde{f}_4]^T$. 
We present the values for $I_i$ in \cref{tab:numerical_results_for_HC_model_catastrophic_error_why} and observe the values of these integrals match up to 11 or 12 digits, while having opposite signs. 
This is the cause of catastrophic cancellation in the calculation of the sum $I_1+I_2+I_3+I_4$. 
To illustrate the effectiveness of our error estimate while avoiding this issue, we estimate the error using QoIs that isolate only the negative (or alternately, positive) components of the error. 
To this end, we define $\left ( e^{(y)}(T),\zeta_{-}^y\right)$, (the $-$ subscript indicates negative error) by setting $\zeta_{-}^y=[\mathbf{1}^{125},\mathbf{0}^{125},\mathbf{1}^{125},\mathbf{0}^{125}]^T$. 
This isolates the error in both differential variables in the left half of the spatial domain (where we observe the errors with negative sign). 
The results for this case are in \cref{tab:numerical_results_for_HC_model_split_diff_variables_error}, and indicate that the effectivity ratio is close to one, as expected. 
We observe similar results for $\left ( e^{(y)}(T),\zeta_{+}^y\right)$, (the $+$ subscript indicates positive error)  when setting $\zeta_{+}^y=[\mathbf{0}^{125},\mathbf{1}^{125},\mathbf{0}^{125},\mathbf{1}^{125}]^T$ as shown in \cref{tab:numerical_results_for_HC_model_split_diff_variables_error}.


\begin{table}[ht]
    \centering
    \caption{Numerical Results for  DAE in \cref{ex:hc_model_example} using Adjoint DAE (\cref{thm:adj2_qoi_terminal_time_both_diff_alg_vari}) to estimate the error $\eQFT$ where $\zeta^y=[\mathbf{1}^{250},\mathbf{1}^{250}]^T$, $\zeta^z = [\mathbf{0}^{249}]^T$ and again $\eQFT$ where $\zeta^z = [\mathbf{1}^{249}]^T$, $\zeta^y=[\mathbf{0}^{250},\mathbf{0}^{250}]^T$ with spatial grid spacing $\Delta x=0.004$, and $\Delta t = 0.001$.}
    \label{tab:numerical_results_for_HC_model_catastrophic_error}
    \begin{tabular}{ccccc}
        \toprule
         \( T \) & \multicolumn{2}{c}{ \(\eQFT=\left( e^{(y)}(T), \zeta^y \right)\)} & \multicolumn{2}{c}{ \(\eQFT = \left( e^{(z)}(T), \zeta^z \right)\)} \\
        \cmidrule(lr){2-3}
        \cmidrule(lr){4-5}
         & Error Estimate & Effectivity Ratio & Error Estimate & Effectivity Ratio \\
        \midrule
         0.5 & 3.1247e-12 & 1.168 & -2.9255e-02 & 1.008 \\
           1 & 4.1497e-12 & 0.9014 & -3.5010e-02 & 0.9881 \\
           2 & 8.8735e-12 & 1.173 & -2.7558e-02 & 0.9795 \\
           3 & 2.4128e-12 & 1.025 & -1.6760e-02 & 0.9769 \\
        \bottomrule
    \end{tabular}
    
\end{table}

    

\begin{table}[ht]
    \centering
    \caption{Numerical Results for  DAE in \cref{ex:hc_model_example} using \cref{thm:adj2_qoi_terminal_time_both_diff_alg_vari} to investigate the reason behind the catastrophic cancellation occurs in $\eQFT = \left( e^{(y)}(T), \zeta^y \right)$ at $T$ with spatial grid spacing $\Delta x=0.004$, and $\Delta t = 0.001$.}
    \label{tab:numerical_results_for_HC_model_catastrophic_error_why}
    \begin{tabular}{ccc}
        \toprule
         \( T \) & Integrals & Approximate Values \\
        \midrule
        0.5 & $I_1$ & -0.012900368135421083 \\
          & $I_2$ & 0.012900368138223206 \\
          & $I_3$ & -0.012900368138363746  \\
          & $I_3$ & 0.012900368138686344  \\
        \bottomrule
    \end{tabular}
    
\end{table}

\begin{table}[h]
    \centering
    \caption{Numerical Results for  DAE in \cref{ex:hc_model_example} using Adjoint DAE (\cref{thm:adj2_qoi_terminal_time_both_diff_alg_vari} with $\zeta^z=[\mathbf{0}^{249}]^T$) to estimate the error $\left( e^{(y)}(T), \zeta_{-}^y \right)$ where $\zeta_{-}^y=[\mathbf{1}^{125},\mathbf{0}^{125},\mathbf{1}^{125},\mathbf{0}^{125}]^T$ and $\left( e^{(y)}(T), \zeta_{+}^y \right)$ where $\zeta_{+}^y=[\mathbf{0}^{125},\mathbf{1}^{125},\mathbf{0}^{125},\mathbf{1}^{125}]^T$ with spatial grid spacing $\Delta x=0.004$.}
    \label{tab:numerical_results_for_HC_model_split_diff_variables_error}
    \begin{tabular}{cccccc}
        \toprule
        \( \Delta t \) & \( T \) & \multicolumn{2}{c}{Error \(\left( e^{(y)}(T), \zeta_{-}^y \right)\)} &\multicolumn{2}{c}{Error \(\left( e^{(y)}(T), \zeta_{+}^y \right)\)}\\
        \cmidrule(lr){3-4}
        \cmidrule(lr){5-6}
        & & Error estimate & Effectivity Ratio & Error estimate & Effectivity Ratio \\
        \midrule
        0.001 & 0.5 & -2.0443e-02 & 0.9715 & 2.0443e-02 & 0.9715 \\
              & 1 & -2.6109e-02 & 0.9716 & 2.6109e-02 & 0.9716 \\
              & 2 & -2.1295e-02 & 0.9716 & 2.1295e-02 & 0.9716 \\
              & 3 & -1.3026e-02 & 0.9717 & 1.3026e-02 & 0.9717 \\
        \midrule      
       0.002 & 0.5 & -4.0866e-02 & 0.9856 & 4.0866e-02 & 0.9856 \\
               & 1 & -5.2201e-02 & 0.9857 & 5.2201e-02 & 0.9857 \\
               & 2 & -4.2588e-02 & 0.9858 & 4.2588e-02 & 0.9858 \\
               & 3 & -2.6058e-02 & 0.9859 & 2.6058e-02 & 0.9859 \\
        \bottomrule
    \end{tabular}
    
\end{table}

\section{Conclusions}\label{sec:conclusion}
In this paper, we propose a novel technique for \emph{a posteriori} error estimates for numerical solution of DAEs. 
In particular, our technique may be used to assess error for QoIs involving the differential \emph{or} algebraic state variables, and that are cumulative in time, or evaluated at the terminal time.
Furthermore, our methodology may be applied to the semi-explicit index-1 or Hessenberg index-2 DAEs.
Our technique is based on the formulation and solution of an adjoint DAE, which is a linear problem, making it relatively cost effective to implement.
We also present a second technique which is based on an adjoint ODE. 
This second technique, while being conceptually simpler, is more computationally more expensive.
We tested these two techniques numerically for various types of index-1 and Hessenberg index-2 DAE problems: linear, nonlinear, autonomous, non-autonomous, and an example PDAE. In both techniques we have the accurate error estimate and the effectivity ratio close to one, unless there was catastrophic cancellation resulting in both the computed and exact error being extremely small.
The adjoint ODE approach is too costly for a large scale problem such as PDAEs, due to the cost of order reduction for large scale systems. 
However, the adjoint DAE approach works extremely well in this case, without the associated cost. 

There are a number of future directions which arise from this work. One such direction is extending the analysis to partial differential algebraic equations in such a way that the  estimates
 identify the error contribution due to the spatial and temporal discretizations, as was done for a PDE  in \cite{chaudhry2019posteriori}. Another direction is deriving adaptive algorithms based on the \textit{a posteriori} estimates. This requires a careful analysis of the numerical scheme under consideration to quantify the effects of different quadratures, time-stepping choices, projection operators, etc. Such analysis for ODEs and PDEs have been carried out in \cite{chaudhry2015posterioriIMEX,chaudhry2017posteriori,collins2015posteriori,collins2014posteriori, Chaudhry2021}, and in future we aim to utilize these ideas to extend the analysis in this work.


\begin{appendices}

\section{Some Proofs}
\label{sec:app_proofs}

\begin{proof}[Proof of \cref{lem:adj_error}]
Multiplying \cref{eq:diff} by the error term $e^{(y)}$ and \cref{eq:alg} by the error term $e^{(z)}$ and integrating by parts,
    \begin{align}
      \eQCI
       &= \langle -\dot{\phi}^{(y)},y-Y\rangle -\langle \bar{f}_y^T\phi^{(y)},y-Y\rangle - \langle \bar{g}_y^T\phi^{(z)},y-Y\rangle
       - \langle \bar{f}_z^T\phi^{(y)},z-Z\rangle\notag\\
       &\quad\quad- \langle \bar{g}_z^T\phi^{(z)},z-Z\rangle \label{eq:plug_ibp}
    \end{align}
    Considering the first term on the right hand side of \cref{eq:plug_ibp} and applying integration by parts,
    \begin{align}
        \langle -\dot{\phi}^{(y)},y-Y\rangle 
        &=\int_{0}^{T}\left( \dot{\phi}^{(y)}(t),{y}(t)-{Y}(t)\right)dt\notag\\
        &=\sum_{k=0}^{N-1}\int_{t_k}^{t_{k+1}}\left( \dot{\phi}^{(y)}(t),{y}(t)-{Y}(t)\right)dt\notag\\
        &=\sum_{k=0}^{N-1}\int_{t_k}^{t_{k+1}}\left( \phi^{(y)}(t),\dot{y}(t)-\dot{Y}(t)\right)dt\notag\\
        &\quad\quad + \sum_{k=0}^{N-1}\left[ \left(\phi^{(y)}(t_k),y(t_k)-Y(t_k)\right)-\left(\phi^{(y)}(t_{k+1}),y(t_{k+1})-Y(t_{k+1}) \right)\right], \notag\\
        &=\sum_{k=0}^{N-1}\int_{t_k}^{t_{k+1}}\left( \phi^{(y)}(t),\dot{y}(t)-\dot{Y}(t)\right)dt\notag\\
        &\quad\quad +  \left(\phi^{(y)}(t_0),y(t_0)-Y(t_0)\right)-\left(\phi^{(y)}(t_{N}),y(t_{N})-Y(t_{N}) \right), \notag\\
        &=\langle \phi^{(y)},\dot{y}-\dot{Y}\rangle +  \left(\phi^{(y)}(0),y(0)-Y(0)\right)-\left(\phi^{(y)}(T),y(T)-Y(T) \right),\label{eq:result_of_ibp}
    \end{align}
    where we also used the fact that all functions under consideration ($\phi, y, Y$) are continuous on $[0,T]$. Using  \cref{eq:result_of_ibp} in \cref{eq:plug_ibp},
    \begin{align}
        \eQCI
        &= \langle \phi^{(y)},\dot{y}-\dot{Y}\rangle -  \langle \phi^{(y)},\bar{f}_y(y-Y) + \bar{f}_z(z-Z)\rangle - \langle \phi^{(z)},\bar{g}_y(y-Y) + \bar{g}_z(z-Z)\rangle\notag\\
       &\quad\quad + \left(\phi^{(y)}(0),y(0)- Y(0)\right) - \left(\phi^{(y)}(T),y(T)-Y(T)\right)\label{eq:similar_approach_error_index_1}.
    \end{align}
    Now using the properties \cref{eq:property_adj_1_property_1} and \cref{eq:property_adj_1_property_2} followed by using \cref{eq:dae_equation},
    \begin{align*}
        \eQCI
        &= \langle \phi^{(y)},\dot{y}-\dot{Y}\rangle -  \langle \phi^{(y)},f(y,z) - f(Y,Z)\rangle - \langle \phi^{(z)},g(y,z) - g(Y,Z)\rangle\\
        &\quad\quad+ \left(\phi^{(y)}(0),y(0)- Y(0)\right)
         - \left(\phi^{(y)}(T),y(T)-Y(T)\right),\\
        &= \langle \phi^{(y)},\dot{y} - f(y,z)\rangle - \langle \phi^{(y)},\dot{Y} - f(Y,Z)\rangle + \langle \phi^{(z)},g(Y,Z)\rangle + \left(\phi^{(y)}(0),y(0)- Y(0)\right)\\
        &\quad\quad - \left(\phi^{(y)}(T),y(T)-Y(T)\right),\\
         &= \left( \phi^{(y)}(0),e^{(y)}(0) \right) - \left(\phi^{(y)}(T),y(T)-Y(T)\right)
         + \langle \phi^{(y)}, f(Y,Z) - \dot{Y}  \rangle+ \langle \phi^{(z)}, g(Y,Z)  \rangle.
    \end{align*}
\end{proof}

\begin{proof}[Proof of \cref{lem:tech_results}]

\begin{enumerate}
    \item 
    \begin{align*}
        \left(P^T\zeta^y,e^{(y)} \right) &= \left(\left(I -\bar{g}^T_{y}(\bar{f}^T_{z}\bar{g}^T_{y})^{-1}\bar{f}_z^T \right)\zeta^y,e^{(y)} \right),\\
        &= \left(\zeta^y,e^{(y)} \right) - \left(\bar{g}^T_{y}(\bar{f}^T_{z}\bar{g}^T_{y})^{-1}\bar{f}_z^T\zeta^y,e^{(y)} \right),\\
        &= \left(\zeta^y,e^{(y)} \right) - \left((\bar{f}^T_{z}\bar{g}^T_{y})^{-1}\bar{f}_z^T\zeta^y,\bar{g}_{y}e^{(y)} \right),\\
        &= \left(\zeta^y,e^{(y)} \right) + \left((\bar{f}^T_{z}\bar{g}^T_{y})^{-1}\bar{f}_z^T\zeta^y,g(Y)\right).
    \end{align*}

    \item 
    \begin{align*}
        &\left(-P^T\bar{f}^T_{y} \bar{g}^T_{y}(\bar{f}^T_{z}\bar{g}^T_{y})^{-1}\zeta^z , e^{(y)} \right),\\
        &= \left(-\left(I -\bar{g}^T_{y}(\bar{f}^T_{z}\bar{g}^T_{y})^{-1}\bar{f}_z^T \right)\bar{f}^T_{y} \bar{g}^T_{y}(\bar{f}^T_{z}\bar{g}^T_{y})^{-1}\zeta^z , e^{(y)} \right),\\
        &= \left(-\bar{f}^T_{y}\bar{g}^T_{y}(\bar{f}^T_{z}\bar{g}^T_{y})^{-1}\zeta^z ,e^{(y)} \right) + \left( \bar{g}^T_{y}(\bar{f}^T_{z}\bar{g}^T_{y})^{-1}\bar{f}_z^T\bar{f}^T_{y}\bar{g}^T_{y}(\bar{f}^T_{z}\bar{g}^T_{y})^{-1}\zeta^z,e^{(y)} \right),\\
        &= -\left(\bar{g}^T_{y}(\bar{f}^T_{z}\bar{g}^T_{y})^{-1}\zeta^z ,\bar{f}_{y}e^{(y)} \right) + \left((\bar{f}^T_{z}\bar{g}^T_{y})^{-1}\bar{f}_z^T\bar{f}^T_{y}\bar{g}^T_{y}(\bar{f}^T_{z}\bar{g}^T_{y})^{-1}\zeta^z,\bar{g}_{y}e^{(y)} \right),\\
        &= -\left(\bar{g}^T_{y}(\bar{f}^T_{z}\bar{g}^T_{y})^{-1}\zeta^z ,f(y,z) - f(Y,Z) - \bar{f}_z(z-Z)\right)\\ 
        &\qquad+\left((\bar{f}^T_{z}\bar{g}^T_{y})^{-1}\bar{f}_z^T\bar{f}^T_{y}\bar{g}^T_{y}(\bar{f}^T_{z}\bar{g}^T_{y})^{-1}\zeta^z,g(y)-g(Y)\right),\\
        &= -\left(\bar{g}^T_{y}(\bar{f}^T_{z}\bar{g}^T_{y})^{-1}\zeta^z ,f(y,z)\right) + \left(\bar{g}^T_{y}(\bar{f}^T_{z}\bar{g}^T_{y})^{-1}\zeta^z ,f(Y,Z)\right) \\
        &\qquad +\left(\bar{g}^T_{y}(\bar{f}^T_{z}\bar{g}^T_{y})^{-1}\zeta^z ,\bar{f}_z(z-Z)\right) -\left((\bar{f}^T_{z}\bar{g}^T_{y})^{-1}\bar{f}_z^T\bar{f}^T_{y}\bar{g}^T_{y}(\bar{f}^T_{z}\bar{g}^T_{y})^{-1}\zeta^z,g(Y)\right),\\
        &= -\left(\bar{g}^T_{y}(\bar{f}^T_{z}\bar{g}^T_{y})^{-1}\zeta^z ,f(y,z)\right) + \left(\bar{g}^T_{y}(\bar{f}^T_{z}\bar{g}^T_{y})^{-1}\zeta^z ,f(Y,Z)\right) +\left(\zeta^z ,e^{(z)} \right)\\
        &\qquad
        - \left((\bar{f}^T_{z}\bar{g}^T_{y})^{-1}\bar{f}_z^T\bar{f}^T_{y}\bar{g}^T_{y}(\bar{f}^T_{z}\bar{g}^T_{y})^{-1}\zeta^z,g(Y)\right).\\
    \end{align*}

    \item \begin{align*}
        &\left( \frac{d \bar{g}^T_{y}}{dt}\left(\bar{f}^T_{z}\bar{g}^T_{y}\right)^{-1} \zeta^z,e^{(y)} \right),\\
        &= \left(\left(\bar{f}^T_{z}\bar{g}^T_{y}\right)^{-1} \zeta^z,\frac{d \bar{g}_{y}}{dt}e^{(y)} \right),\\
        &= \left(\left(\bar{f}^T_{z}\bar{g}^T_{y}\right)^{-1} \zeta^z,\frac{d [\bar{g}_{y}e^{(y)} ]}{dt}-\bar{g}_{y}\frac{d e^y }{dt}\right),\\
        &= \left(\left(\bar{f}^T_{z}\bar{g}^T_{y}\right)^{-1} \zeta^z,-\frac{d g(Y)}{dt}-\bar{g}_{y}[f(y,z)-\dot{Y}]\right),\\
        &= - \left(\left(\bar{f}^T_{z}\bar{g}^T_{y}\right)^{-1} \zeta^z,\frac{d g(Y)}{dt}\right) - \left(\left(\bar{f}^T_{z}\bar{g}^T_{y}\right)^{-1} \zeta^z,\bar{g}_{y}f(y,z)\right)
        + \left(\left(\bar{f}^T_{z}\bar{g}^T_{y}\right)^{-1} \zeta^z,\bar{g}_{y}\dot{Y}\right),\\
        &= -\left(\left(\bar{f}^T_{z}\bar{g}^T_{y}\right)^{-1} \zeta^z,\frac{d g(Y)}{dt}\right) - \left(\bar{g}^T_{y}\left(\bar{f}^T_{z}\bar{g}^T_{y}\right)^{-1} \zeta^z,f(y,z)\right)+ \left(\bar{g}^T_{y}\left(\bar{f}^T_{z}\bar{g}^T_{y}\right)^{-1} \zeta^z,\dot{Y}\right).\\
    \end{align*}
    
\end{enumerate}
\end{proof}

\section{Staggered Grid Discretization}
\label{macscheme}
We use relatively standard central difference approximations for the spatial derivative of $c_x,c,a_x,$ and $a$. 
The nonlinear terms $ cw$, and $aw$, require multiplication of quantities which ``live'' at cell centers and those that live at cell edges. 
In this case, we average the two cell centered quantities ($c$ or $a$) to the cell edge between. 
We explicitly show the details of the spatial discretization for $c$ below. 
The discretization of $a$ is completely analogous.

The partial derivative of $c$ with respect to $x$ at $\tilde{x}_j$ is approximated as follows:
\[
c_x(\tilde x_j) \approx \frac{c(\tilde{x}_j+\Delta x/2) - c(\tilde{x}_j-\Delta x/2)}{\Delta x} = \frac{c(x_{j+1}) - c(x_j)}{\Delta x}. 
\]
That is to say that first derivatives (which naturally live at cell edges) are calculated by applying centered differences to quantities defined at cell centers. 
For the nonlinear terms, we average concentrations to the cell edges to yield 
\begin{equation}
    c(\tilde{x}_j) w(\tilde{x_j}) \approx \frac{c(\tilde{x}_j+\Delta x/2) + c(\tilde{x}_j-\Delta x/2)}{2}w(\tilde{x}_j) = \frac{c(x_{j+1}) + c(x_j)}{2}w(\tilde{x}_j)
\end{equation}  
We can now apply centered differences to the above expressions to yield the following expressions for $j=2,3,\dots,N_s-1$,
\begin{align*}
    &\left. \frac{\partial}{\partial x}\left(c_{x}+  cw\right) \right|_{x_j}\\
    &\approx \frac{1}{\Delta x}\left(c_x(x_j+\Delta x/2)+c(x_j+\Delta x/2)w(x_j+\Delta x/2)\right)\\
    &\quad\quad - \frac{1}{\Delta x}\left(c_x(x_j-\Delta x/2)+c(x_j-\Delta x/2)w(x_j-\Delta x/2)\right),\\
    &= \frac{1}{\Delta x}\left(c_x(\tilde{x}_j)+c(\tilde{x})w(\tilde{x}_j)\right)
     - \frac{1}{\Delta x}\left(c_x(\tilde{x}_{j-1})+c(\tilde{x}_{j-1})w(\tilde{x}_{j-1})\right),\\
    &= \frac{1}{\Delta x}\left(\frac{c(x_{j+1}) - c(x_j)}{\Delta x}+\frac{c(x_{j+1}) + c(x_j)}{2}w(\tilde{x}_j)\right)\\
    &\quad \quad - \frac{1}{\Delta x}\left(\frac{c(x_{j}) - c(x_{j-1})}{\Delta x}+\frac{c(x_{j}) + c(x_{j-1})}{2}w(\tilde{x}_{j-1})\right)\\
    &= \frac{1}{\Delta x^2}\left(c(x_{j+1})-2c(x_j)+c(x_{j-1})\right)\\
    &\quad\quad+\frac{1}{2\Delta x}\Big(\left(c(x_j)+c(x_{j+1})\right)w(\tilde{x}_j)-\left(c(x_j)+c(x_{j-1})\right)w(\tilde{x}_{j-1})\Big).
\end{align*}

For $j = 1$ and $j = N_s$, we need to utilize the  no flux (homogeneous Robin) boundary conditions \cref{eq:robin_bcs}.
These can be expressed as 
\[ \left. c_{x}+  cw \right|_{x=0} = \left. c_{x}+  cw \right|_{\tilde{x}_0} = 0, \]
and 
\[ \left. c_{x}+  cw \right|_{x=1} = \left. c_{x}+  cw \right|_{\tilde{x}_{N_s}} = 0. \]
Now, for $j=1$, i.e., the scheme at the left boundary of the modified domain is
\begin{align*}
    &\left. \frac{\partial}{\partial x}\left(c_{x}+  cw\right) \right|_{x_1}\\
    &\approx \frac{1}{\Delta x}\left(\Big(c_x(\tilde{x}_1)+c(\tilde{x}_1)w(\tilde{x}_1)\Big) - \Big(c_x(\tilde{x}_0)+c(\tilde{x}_0)w(\tilde{x}_0)\Big)\right),\\
    &= \frac{1}{\Delta x}\left(c_x(\tilde{x}_1)+c(\tilde{x}_1)w(\tilde{x}_1)\right)-  0 ,\\
    &= \frac{1}{\Delta x}\left(\frac{c(x_{2}) - c(x_1)}{\Delta x}+\frac{c(x_{2}) + c(x_1)}{2}w(\tilde{x}_1)\right)\\
    &= \frac{1}{\Delta x^2}\left(c(x_{2})-c(x_1))\right) +\frac{1}{2\Delta x}\Big(\left(c(x_1)+c(x_{2})\right)w(\tilde{x}_1)\Big).
\end{align*}
Similarly, for $j=N_s$, i.e., the scheme at the right boundary of the modified domain is 
\begin{align*}
    &\left. \frac{\partial}{\partial x}\left(c_{x}+  cw\right) \right|_{x_{N_s}}\\
    &\approx \frac{1}{\Delta x}\left(\Big(c_x(\tilde{x}_{N_s})+c(\tilde{x}_{N_s})w(\tilde{x}_{N_s})\Big) - \Big(c_x(\tilde{x}_{N_s-1})+c(\tilde{x}_{N_s-1})w(\tilde{x}_{N_s-1})\Big)\right),\\
    &= 0 - \frac{1}{\Delta x}\left(c_x(\tilde{x}_{N_s-1})+c(\tilde{x}_{N_s-1})w(\tilde{x}_{N_s-1})\right),\\
    &= - \frac{1}{\Delta x}\left(\frac{c(x_{N_s}) - c(x_{N_s-1})}{\Delta x}+\frac{c(x_{N_s}) + c(x_{N_s-1})}{2}w(\tilde{x}_{N_s-1})\right)\\
    &= \frac{1}{\Delta x^2}\left( -c(x_{N_s})+c(x_{N_s-1})\right)+\frac{1}{2\Delta x}\Big(-\left(c(x_{N_s})+c(x_{N_s-1})\right)w(\tilde{x}_{N_s-1})\Big).
\end{align*}
A completely analogous discretization is used for the the anion $a$. 
These finite difference discretizations are used to construct the right hand side of the differential equation for the differential variables in \cref{eq:hc_index_2_equation}. 

In order to construct the algebraic constraint for \cref{eq:hc_index_2_equation} we notice that the above finite difference equations are conservative. 
This implies that 
\[ \sum_{j= 1}^{N_s} c(x_j,t) \text{ and } \sum_{j = 1}^{N_s} a(x_j,t),\]
are both constants independent of time. 
This, in turn, means that if $c(x_j,0) - a(x_j,0) = 0$ at the initial time, and $c(x_j) - a(x_j) = 0$ for $j = 2, 3, \ldots N_s$, then it must be true that $c(x_{1}) - a(x_{1}) = 0$. 
Put another way, if we use a conservative discretization and use initial conditions which satisfy the constraint at all spatial locations, then we do not need to explicitly enforce the constraint at all locations. 
We need only enforce it in $N_s -1$ grid cells, and we are guaranteed that it will be satisfied in the remaining grid cell ``for free''. 
Thus, we do not explicitly require that $C - A = 0$ in \cref{eq:hc_index_2_equation}. 
Rather, we require that $\Pi(C - A) = 0$, where $\Pi$ is the linear operator which simply evaluates a variable at all but the first cell center.

\end{appendices}

\bibliography{dae1}

\end{document}